\documentclass[12pt]{article}
\usepackage{amsmath,amssymb,amstext,dsfont,fancyvrb,float,fontenc,graphicx,subfigure,theorem,hyperref}

\usepackage{tikz,everypage}

\usepackage[letterpaper]{geometry}
\ifx\volno\undefined\def\volno{0}\fi
\ifx\volyear\undefined\def\volyear{2017}\fi
\ifx\pagno\undefined\def\pagno{000--000}\fi

\newfont{\footsc}{cmcsc10 at 8truept}
\newfont{\footbf}{cmbx10 at 8truept}
\newfont{\footrm}{cmr10 at 10truept}

\usepackage{fancyhdr}
\usepackage{relsize}
\usepackage{sectsty}
\allsectionsfont{\larger[-1]} 

\renewcommand\paragraph{\@startsection{paragraph}{4}{\z@}
                                    {2ex \@plus.5ex \@minus.2ex}
                                    {-1em}
                                    {\normalfont\normalsize\bfseries}}

\renewcommand\subparagraph{\@startsection{subparagraph}{5}{\parindent}
                                       {2ex \@plus.5ex \@minus .2ex}
                                       {-1em}
                                      {\normalfont\normalsize\bfseries}}

\newlength{\BiblioSpacing}
\renewenvironment{thebibliography}[1]{
\begin{oldthebibliography}{#1}
\setlength{\parskip}{\BiblioSpacing}
\setlength{\itemsep}{\BiblioSpacing}
}
{
\end{oldthebibliography}
}

\usepackage[strict]{changepage}
\def\abstractname{Abstract -}   
\def\abstract{\begin{adjustwidth}{1cm}{1cm} \par    \footnotesize \noindent {\bf \abstractname} 
\def\endabstract{ \end{adjustwidth} \smallskip }}

\DeclareMathOperator{\sinc}{sinc}

{\theorembodyfont{\itshape}\newtheorem{theorem}{Theorem}[section]}
{\theorembodyfont{\itshape}\newtheorem{proposition}[theorem]{Proposition}}
{\theorembodyfont{\itshape}}
{\theorembodyfont{\itshape}\newtheorem{lemma}[theorem]{Lemma}}
{\theorembodyfont{\itshape}\newtheorem{corollary}[theorem]{Corollary}}
{\theorembodyfont{\rm}}
{\theorembodyfont{\rm}\newtheorem{remark}[theorem]{Remark}}
{\theorembodyfont{\rm}\newtheorem{conjecture}[theorem]{Conjecture}}
{\theorembodyfont{\rm }}

\usepackage{csquotes} 

\title{\Large\bf Counting Restricted Partitions of Integers into Fractions: Symmetry and Modes of the Generating Function and a Connection to $\omega(t)$}
\author{\sc Z. Hoelscher and E. Palsson \footnote{The second author was supported in part by Simons Foundation Grant \#360560.}}

\numberwithin{equation}{subsection}
\numberwithin{theorem}{subsection}

\begin{document}
\setcounter{page}{1}
\maketitle
\thispagestyle{fancy}

\vskip 1.5em

\begin{abstract}
Motivated by the study of integer partitions, we consider partitions of integers into fractions of a particular form, namely with constant denominators and distinct odd or even numerators. When numerators are odd, the numbers of partitions for integers smaller than the denominator form symmetric patterns. If the number of terms is restricted to $h$, then the nonzero terms of the generating function are unimodal, with the integer $h$ having the most partitions. Such properties can be applied to a particular class of nonlinear Diophantine equations. We also examine partitions with even numerators. We prove that there are $2^{\omega(t)}-2$ partitions of an integer $t$ into fractions with the first $x$ consecutive even integers for numerators and equal denominators of $y$, where $0<y<x<t$. We then use this to produce corollaries such as a Dirichlet series identity and an extension of the prime omega function to the complex plane, though this extension is not analytic everywhere.
\end{abstract}
 

\begin{MSC}
05A17
\end{MSC}

\section{Introduction}

Integer partitions are a classic part of number thoery. In the most general, unrestricted case, one seeks to express positive integers as the sum of smaller positive integers. Often the function $p(n)$ is used to denote the count of the partitions of $n$. No simple formula for this is known, though a generating function can be written \cite{one}. We note that a generating function is a formal power series where the exponent of $x$ in each term refers to the number being partitioned, and the coefficient of $x$ gives the number of partitions.
\begin{equation}
     \sum_{n=0}^{\infty} p(n)x^n=\displaystyle\prod_{m\geq1} \frac{1}{1-x^m}
\end{equation}

Ramanujan produced congruence formulas for $p(n)$ \cite{two}, while Rademacher wrote a series expansion that allowed for asymptotic bounds \cite{one}. Note that $\omega(h,k)$ is the $24k^{th}$ root of unity. 
\begin{equation}
    p(n)=\frac{1}{\pi \sqrt{2}}\sum_{k=1}^{\infty}\sqrt{k} \sum_{\substack {0\leq h<k \\ (h,k)=1}}\omega(h,k) e^{\frac{-2\pi i nh}{k}} \frac{d}{dn}\left(\frac{\sinh\left(\frac{\pi}{k}\sqrt{\frac{2}{3}(n-\frac{1}{24})}\right)}{\sqrt{n-\frac{1}{24}}}\right)
\end{equation}

Much work has been done on restricted partitions. For example, we can denote the number of partitions into distinct parts as $q(n)$, where the generating function is given below \cite{three}.
\begin{equation}
    \sum_{n=0}^{\infty} q(n)x^n=\displaystyle\prod_{j=1}^{\infty}(1+x^j)
\end{equation}
It is known that Rademacher-like series exist for this case as well \cite{three}. It is also common to find partitions into odd integers, which we count with $p_{O}(n)$, or partitions into distinct odd terms, which we count with $q_{O}(n)$. Both have been studied by Hagis \cite{four}, where he writes the generating function for $q_{O}(n)$ as 
\begin{equation}
    1+\sum_{n=1}^{\infty} q_{O}(n)x^n=\displaystyle\prod_{m=0}^{\infty}(1+x^{2m+1})
\end{equation}

Another common constraint is to restrict the value of the largest summand to the integer $s$. In the general case, we can denote the number of such partitions as $p_s(n)$, and the case with distinct summands as $q_s(n)$. Upper bounds for $q_s(n)$ have been found by Bidar \cite{three}. Alternatively, one can restrict partitions to $h$ parts, where the number of partitions is counted by $p_{h}(n)$, or $q_{h}(n)$, where we also require that the summands are distinct. A recurrence relation for $q_{h}(n)$ immediately follows from the work of Christopher, where we specify that there are $h$ different sizes for the $h$ parts, as they are distinct \cite{five}.

\begin{equation}
    q_{h}(n)=q_h(n-h)+q_{h-1}(n-h)
\end{equation}

One can also combine many of these restrictions. We let $q_{O_s}(n)$ count partitions into odd distinct summands no larger than $s$, and $q_{O_s,h}(n)$
count partitions into $h$ distinct odd summands no larger than $s$. The latter is the more interesting of the two. We can write the following generating functions
\begin{equation}
    \sum_{m=0}^{j^2} q_{O_{2j-1}}(m)x^m=\displaystyle\prod_{n=1}^{j}(1+x^{2n-1})
\end{equation}
\begin{equation}
    \sum_{m=0}^{j^2} \sum_{h=0}^{j} q_{O_{2j-1},h}(m) y^hx^m=\displaystyle\prod_{n=1}^{j} (1+yx^{2n-1})
\end{equation}

In this paper we study partitions of integers into fractions with fixed denominators, which is equivalent to further restrictions on the integer partition questions given above. It is a well-known fact that $j^2=\sum_{n=1}^{j} (2n-1)$, where $j$ is a positive integer. If we then divide both sides by $j$, we find the series 
\begin{equation}
    j=\frac{1}{j} + \frac{3}{j} + \cdots + \frac{2(j-1)-1}{j}+\frac{2j-1}{j}
\end{equation}
It is natural to ask whether there exist distinct coefficients $a_1,\ldots,a_n \in {1,3,\ldots,2j-1}$ such that any positive integer $k$, where $0< k < j$ can be written as 
\begin{equation}
    k = \frac{a_1}{j} + \frac{a_2}{j}+\ldots +\frac{a_{n-1}}{j}+\frac{a_n}{j}
\end{equation}
If $j>2$, this is, in fact, always possible. We provide a proof in Section \ref{sec:4}. We note that this is not always possible for other types of fractions, such as those with even numerators. An interesting continuation of this work would be to determine all infinite sequences $\{ n_j \}$ of numerators that enable such solutions for all $k<\lfloor G \rfloor$, where $G=\frac{n_1}{j} + \frac{n_2}{j} + \cdots + \frac{n_{j-1}}{j}+\frac{n_j}{j}$, $k,j \in \mathbb{N}$, $j>2$. For instance, one could say that $\{ n_j \}$ is the Fibonacci sequence instead of the odd integers. Note that for alternative choices of $\{ n_j \}$, $\lfloor G \rfloor$ may no longer equal $j$. If an infinite number of such sequences exist, one could look for necessary and sufficient conditions for a sequence to qualify. We give an example below for the case of $j=6$:

\begin {equation*}
6=\frac{1}{6} + \frac{3}{6} + \frac{5}{6} + \frac{7}{6} + \frac{9}{6} + \frac{11}{6}
\end {equation*}

\begin {equation*}
1=\frac{1}{6} + \frac{5}{6} \hspace{0.2 in} 2=\frac{11}{6} + \frac{1}{6} \hspace{0.2 in} 3=\frac{7}{6} + \frac{11}{6} \hspace{0.2 in} 4=\frac{11}{6} + \frac{9}{6} + \frac{3}{6} + \frac{1}{6} \hspace{0.2 in} 5=\frac{11}{6} + \frac{9}{6} + \frac{7}{6} + \frac{3}{6}
\end {equation*}

\begin{conjecture}
The Lazy Caterer's sequence (A000124) enables solutions for all $k<\lfloor G \rfloor$, $j>2$, $k,j \in \mathbb{N}$. This also appears to be true for the Cake Numbers (A000125).
\end{conjecture}

It is more interesting to ask how many solutions there are, depending on the values of $k$ and $j$. We represent this count with the function $f_{O_j}(k)$. 

\begin{table}[!htbp]
  \centering
    \begin{tabular}{ |c c|c|c|c|c|c|c|c|c|c|c|c|}
    \hline
    \multicolumn{13}{|c|}{Values for $f_{O_j}(k)$} \\
    \hline
    $j$ $\hspace{0.2 cm} | $& $k$ & $1$ & $2$ & $3$ & $4$ & $5$ & $6$ & $7$ & $8$ & $9$ & $10$ & $11$\\
    \hline\hline
    3 & $ $ & \textbf{1} & \textbf{1} & 0 & 0 & 0 & 0 & 0 & 0 & 0 & 0 & 0\\
    
    4 & $ $ & \textbf{1} & \textbf{2} & \textbf{1} & 0 & 0 & 0 & 0 & 0 & 0 & 0 & 0\\
    
    5 & $ $ & \textbf{1} & \textbf{2} & \textbf{2} & \textbf{1} & 0 & 0 & 0 & 0 & 0 & 0 & 0\\
    
    6 & $ $ & \textbf{1} & \textbf{3} & \textbf{2} & \textbf{3} & \textbf{1} & 0 & 0 & 0 & 0 & 0 & 0\\
    
    7 & $ $ & \textbf{1} & \textbf{3} & \textbf{5} & \textbf{5} & \textbf{3} & \textbf{1} & 0 & 0 & 0 & 0 & 0\\
    
    8 & $ $ & \textbf{2} & \textbf{5} & \textbf{7} & \textbf{8} & \textbf{7} & \textbf{5} & \textbf{2} & 0 & 0 & 0 & 0\\
   
    9 & $ $ & \textbf{2} & \textbf{5} & \textbf{9} & \textbf{13} & \textbf{13} & \textbf{9} & \textbf{5} & \textbf{2} & 0 & 0 & 0\\
  
    10 & $ $ & \textbf{2} & \textbf{7} & \textbf{12} & \textbf{20} & \textbf{20} & \textbf{20} & \textbf{12} & \textbf{7} & \textbf{2} & 0 & 0\\
    
    11 & $ $ & \textbf{2} & \textbf{8} & \textbf{18} & \textbf{29} & \textbf{36} & \textbf{36} & \textbf{29} & \textbf{18} & \textbf{8} & \textbf{2} & 0\\
    
    12 & $ $ & \textbf{3} & \textbf{11} & \textbf{25} & \textbf{44} & \textbf{60} & \textbf{68} & \textbf{60} & \textbf{44} & \textbf{25} & \textbf{11} & \textbf{3}\\
    \hline
    \end{tabular}
\end{table}

\newpage

It is even more interesting to ask what happens when the number of terms is restricted to $h$. We denote the number of such solutions with $f_{O_j,h}(k)$. One should note that in this case, solutions need not always exist. Below we give a number table corresponding to the case of $h=2$.

\begin{table}[h!]
  \centering
    \begin{tabular}{ |c c|c|c|c|c|c|c|c|c|c|c|c|c|}
    \hline
    \multicolumn{14}{|c|}{Values for $f_{O_j,2}(k)$} \\
    \hline
    $j$ $\hspace{0.2 cm} | $& $k$ & $1$ & $2$ & $3$ & $4$ & $5$ & $6$ & $7$ & $8$ & $9$ & $10$ & $11$ & $12$\\
    \hline\hline
    3 & $ $ & 0 & \textbf{1} & 0 & 0 & 0 & 0 & 0 & 0 & 0 & 0 & 0 & 0\\
    
    4 & $ $ & \textbf{1} & \textbf{2} & \textbf{1} & 0 & 0 & 0 & 0 & 0 & 0 & 0 & 0 & 0\\
    
    5 & $ $ & 0 & \textbf{2} & 0 & 0 & 0 & 0 & 0 & 0 & 0 & 0 & 0 & 0\\
    
    6 & $ $ & \textbf{1} & \textbf{3} & \textbf{1} & 0 & 0 & 0 & 0 & 0 & 0 & 0 & 0 & 0\\
    
    7 & $ $ & 0 & \textbf{3} & 0 & 0 & 0 & 0 & 0 & 0 & 0 & 0 & 0 & 0\\
    
    8 & $ $ & \textbf{2} & \textbf{4} & \textbf{2} & 0 & 0 & 0 & 0 & 0 & 0 & 0 & 0 & 0\\
   
    9 & $ $ & 0 & \textbf{4} & 0 & 0 & 0 & 0 & 0 & 0 & 0 & 0 & 0 & 0\\
  
    10 & $ $ & \textbf{2} & \textbf{5} & \textbf{2} & 0 & 0 & 0 & 0 & 0 & 0 & 0 & 0 & 0\\
    
    11 & $ $ & 0 & \textbf{5} & 0 & 0 & 0 & 0 & 0 & 0 & 0 & 0 & 0 & 0\\
    
    12 & $ $ & \textbf{3} & \textbf{6} & \textbf{3} & 0 & 0 & 0 & 0 & 0 & 0 & 0 & 0 & 0\\
   
    13 & $ $ & 0 & \textbf{6} & 0 & 0 & 0 & 0 & 0 & 0 & 0 & 0 & 0 & 0\\
    \hline
    \end{tabular}
\end{table}

Both of these questions are really integer partition questions in disguise, as we are counting partitions of $kj$ into distinct odd integers no larger than $2j-1$, and, in the more restricted case, only $h$ terms. This of course means that $f_{O_j}(k)=q_{O_{2j-1}}(kj)$ and $f_{O_j,h}(k)=q_{O_{2j-1},h}(kj)$. 

The goal of this paper is to understand the behavior of the functions $f_{O_j}(k)$ and $f_{O_j,h}(k)$ beyond knowledge of their generating functions. This is challenging, as there appears to be no work on this in the literature. Using $q_s(n)$ from Bidar's work, we begin with the following generating function \cite{three}. 
\begin{equation}
    q_s(x)=\displaystyle\prod_{n=1}^{s} (1+x^{n})=\sum_{m=0}^{\frac{s(s+1)}{2}}q_s(m)x^m
\end{equation}

Bidar writes that $q_s(x)$ is a symmetric, unimodal polynomial but says that proving it is unimodal is quite difficult. Here a unimodal polynomial is a polynomial whose coefficients strictly increase to some maximum value, then strictly decrease. A polynomial with multiple modes then has multiple peaks in the values of its coefficients. Bidar is not aware of an elementary proof, with the only one he knows of requiring the use of Lie algebras. In a similar fashion, we can consider the generating functions for $f_{O_j}(k)$ and $f_{O_j,h}(k)$. The following generating functions are polynomials in $x$, just like $q_s(x)$, hence leading us to the main theorem of this paper. Note that $0<k<j$. 
\begin{equation}
    r(x)=f_{O_j}(1)x^1+f_{O_j}(2)x^2+\cdots+f_{O_j}(j-2)x^{j-2}+f_{O_j}(j-1)x^{j-1}
\end{equation}
\begin{equation}
    r_h(x)=f_{O_j,h}(1)x^1+f_{O_j,h}(2)x^2+\cdots+f_{O_j,h}(j-2)x^{j-2}+f_{O_j,h}(j-1)x^{j-1}
\end{equation}

\begin{theorem}
\label{thm1}
The nonzero terms of $r_h(x)$ are unimodal and symmetric about the term that corresponds to $k=h$, where the maximum coefficient is $f_{O_j,h}(h)$.
\end{theorem}

\begin{proposition}
\label{prop1}
The terms of $r(x)$ exhibit symmetry such that the coefficients follow $f_{O_j}(k_1)=f_{O_j}(j-k_1)$. Note that $k_1 \in \mathbb{N}, k_1<j$.
\end{proposition}

We cannot make a similar statement of unimodality for $r(x)$, as the case of $j=6$ is bimodal. This corresponds to the fourth row in the table for $f_{O_j}(k)$. We also see that unlike $r_h(x)$, the greatest coefficient of $r(x)$ with a given $j$ can sometimes be found for more than one consecutive $k$-value. For example, the eighth row  of the table for $f_{O_j}(k)$ has three consecutive entries of 20, whereas with $r_h(x)$, the greatest coefficient is only for a single value of $k$, where $k=h$. These changes in behavior perhaps make our restrictions more interesting.  

We are able to produce a simple proof of unimodality when the number of terms is restricted to $h$, as this restriction enables a simple bijection to another problem, where one can ultimately show that unimodality is a consequence of the unimodality of the Gaussian binomial coefficients. This is not possible for $r(x)$, as the number of terms per partition is not restricted. We also note that $r(x)$ only appears to follow our conjecture below when $k$ is an integer. If $k$ is not required to be an integer, one can easily find instances where $r(x)$ is neither unimodal nor bimodal, such as when $j=11$. Considering this along with the fact that the unimodality of Gaussian binomial coefficients is difficult to prove, our following conjecture may prove to be a fairly difficult, interesting problem. 

\begin{conjecture}
$r(x)$ is always either unimodal or bimodal. 
\end{conjecture}

In contrast to the difficulties faced when dealing with $q_s(x)$, our proof of the unimodality of nonzero terms of $r_h(x)$ is fairly elementary, at least in the sense that it requires only elementary methods and previous results that have combinatorial proofs. We prove this in Section \ref{sec:5}. We also explore simplifications of $r_h(x)$ in Section \ref{sec:2}. In Section \ref{sec:3} we explore closed-form expressions for the case of $h=2$, which we prove in Section \ref{sec:7}. In Section \ref{sec:8} we provide examples of number tables beyond the case of restriction to two terms, and in Section \ref{sec:9} we include further computational results. In Section \ref{sec:10} we connect the odd case to finding solutions to a certain nonlinear Diophantine equation.

After examining such fractions with odd numerators, it is natural to investigate a similar problem with even numerators. 

\begin{equation}
    t=\frac{2}{y} + \frac{4}{y} + \cdots + \frac{2(x-1)}{y}+\frac{2x}{y}
\end{equation}
We observe that if one allows the denominators to be values other than $t$, and the length of the series to vary, then there are often multiple series for each $t$. We also note that in some cases, though not all, these series can be partitioned into all natural numbers $k$, where $0<k<t$. Examples of such partitions are included in Section \ref{sec:9}. We include below a few examples of series for $t$.

\newpage

\begin{displaymath}
6=\frac{2}{2}+\frac{4}{2}+\frac{6}{2}
\end{displaymath}
\begin{displaymath}
10=\frac{2}{2}+\frac{4}{2}+\frac{6}{2}+\frac{8}{2}
\end{displaymath}
\begin{displaymath}
10=\frac{2}{3}+\frac{4}{3}+\frac{6}{3}+\frac{8}{3}+\frac{10}{3}
\end{displaymath}
\begin{displaymath}
14=\frac{2}{4}+\frac{4}{4}+\frac{6}{4}+\frac{8}{4}+\frac{10}{4}+\frac{12}{4}+\frac{14}{4}
\end{displaymath}
This case leads to a nice expression for the count of solutions, as well as a connection to the prime omega function. We provide proofs for the following results in Section \ref{sec:6}.

\begin{theorem}
\label{thm2}
Let $F_E(t)$ be the number of partitions of the positive integer $t$ into fractions with the first $x$ consecutive even integers as numerators and equal denominators of $y$, such that $x$ and $y$ are positive integers, and $0<y<x<t$. We then have $F_E(t)=2^{\omega(t)}-2$, where $\omega(t)$ is the number of distinct prime factors of $t$.
\end{theorem}

We find that we can use our partition identity to produce a series identity for $\frac{\zeta^2(s)}{\zeta(2s)}$ and a continuation of $\omega(t)$ to non-integers and complex numbers $z$. This is not analytic everywhere, though it is analytic in many places, such as $z \in \mathbb{R}, n<z<n+1, n \in \mathbb{N}.$ This can be used to \enquote{assign} a quantity of distinct prime factors to numbers that obviously have none. For example, we can define $\omega(e) \approx -6.0963 + 4.5323i$. Note that we use the normalized $\sinc$ function. This result is of course valid for natural numbers as well, so it can be considered a general formula for the prime omega function. 

\begin{corollary}
\label{cor2}
\begin{equation}
    \omega(z)=\log_2 \bigg( \sum_{x=1}^{\left \lceil{\Re(z)}\right \rceil } \sinc \bigg( \displaystyle\prod_{y=1}^{\left \lceil{\Re(z)}\right \rceil+1}(x^2+x-yz) \bigg) \bigg)
\end{equation}
\end{corollary}

\begin{corollary}
\label{cor3}
Let $s$ be a real number greater than two, then
\begin{equation}
\sum_{t=1}^{\infty} \sum_{x=1}^{t} \frac{\sinc \bigg( \frac{(-t)^{t+1}\Gamma(\frac{-x^2}{t}-\frac{x}{t}+t+2)}{\Gamma(\frac{-x^2}{t}-\frac{x}{t}+1)} \bigg)}{t^s}  = \frac{\zeta^2(s)}{\zeta(2s)}
\end{equation}
\end{corollary}

This is done in part by using a modification of the concept of the circle method. These final corollaries are largely unimportant but interesting curiosities, though the concept of a continuation of a function beyond its usual domain is quite important in number theory. Examples include the analytic continuation of the Riemann zeta function \cite{six} and the gamma function for the factorial.  

\section{Simplifying \texorpdfstring{$r_h(x)$}{}}
\label{sec:2}
We know we can write $r_h(x)$ as the following, where we are only interested in values of $m$ that are integral multiples of $j$, as $m=kj$ when $f_{O_j,h}(k)=q_{O_{2j-1},h}(m)$, where $0<k<j$.
\begin{equation}
    \sum_{m=0}^{j^2} \sum_{h=0}^{j} q_{O_{2j-1},h}(m) y^hx^m=\displaystyle\prod_{n=1}^{j} (1+yx^{2n-1})
\end{equation}

We point out that if we remove from the sum all cases where $q_{O_{2j-1},h}(m)=0$, we see that we can use a relation $R$ to condense the double sum into a single sum. This relation utilizes the Rascal triangle, an alternative to Pascal's triangle that was published in 2010. We let the $n^{th}$ entry, counted from left to right on the $j^{th}$ row be denoted by $T(j,n)$. The values of this triangle are produced using the formula $T(j,n)=\frac{(T(j-1,n-1))(T(j-1,n))+1}{T(j-2,n-1)}$, where the first and last terms of each row are one \cite{seven}. The first few rows are

\begin{center}
1\\
1, 1\\
1, 2, 1\\
1, 3, 3, 1\\
1, 4, 5, 4, 1\\
1, 5, 7, 7, 5, 1\\
1, 6, 9, 10, 9, 6, 1\\
\end{center}
As an example for the recurrence above, we see that for the entry $10$ in the bottom row, we have $10=\frac{(7)(7)+1}{5}$. We see that we can write the generating function
\begin{equation}
     \sum_{0 \leq m \leq j^2}^{}  q_{O_{2j-1},\lambda}(m)y^{\lambda}x^m=\displaystyle\prod_{n=1}^{j} (1+yx^{2n-1})
\end{equation}

When $q_{O_{2j-1},\lambda}(m)$ is nonzero, $R$ is a relation that takes $m$ as its input and returns each integer $\lambda$ from $0$ to $j$ a number of times given by the entries read from left to right on the $(j+1)^{th}$ row of the Rascal triangle. This is then equivalent to saying that the entries in the Rascal triangle give the number of distinct values one can find as sums of a given number of integers removed from the set of the first $j$ odd integers. The ability of the Rascal triangle to predict the number of distinct restricted sums taken from consecutive integers has been noticed before, as indicated by a comment on the OEIS page \cite{eight}, though now we utilize this property to condense a generating function. The $(j+1)^{th}$ row is needed because we have exponents of $y$ from $0$ to $j$.

We see that $R$ does not always return its values in consecutive order and that it can return multiple values for each $m$. $\lambda$ is even if $m$ is even, and odd if $m$ is odd. This is because the sum of an even number of odd integers is even, and the sum of an odd number of odd integers is odd. The index $m$ can range from 0 to $j^2$, though it does not take on every value. We define $\lambda=0$ when $m=0$ and $q_{O_{2j-1},\lambda}(m)=1$ when $m=0$ and $\lambda=0$.

As an example, we display the case of $j=6$:
\begin{equation}
    \prod_{n=1}^{6} (1+yx^{2n-1}) = 
    x^{36}y^6+x^{35}y^5+x^{33}y^5+x^{32}y^4+x^{31}y^5+x^{30}y^4+x^{29}y^5+2x^{28}y^4
\end{equation}
\begin{equation*}
    +x^{27}y^5+x^{27}y^3+2x^{26}y^4+x^{25}y^5+x^{25}y^3+3x^{24}y^4+2x^{23}y^3+2x^{22}y^4+3x^{21}y^3+2x^{20}y^4
\end{equation*}
\begin{equation*}
    +x^{20}y^2+3x^{19}y^3+x^{18}y^4+x^{18}y^2+3x^{17}y^3+x^{16}y^4+2x^{16}y^2+3x^{15}y^3+2x^{14}y^2+2x^{13}y^3
\end{equation*}
\begin{equation*}
    +3x^{12}y^2+x^{11}y^3+x^{11}y+2x^{10}y^2+x^9y^3+x^9y+2x^8y^2+x^7y+x^6y^2+x^5y+x^4y^2+x^3y+xy+1
\end{equation*}

We observe that $y^0$ appears one time, $y^1$ six times, $y^2$ nine times, $y^3$ ten times, $y^4$ nine times, $y^5$ six times, and $y^6$ one time. This matches the seventh row of the Rascal triangle. We see that some values for $m$, such as 34 and 2, do not appear. We also see that the relation returns multiple values for some $m$, as $m=27$ returns both 5 and 3 as exponents of $y$. When the exponent of $x$ is $kj$, the coefficients of the polynomial in this example match the entries in the tables for $f_{O_j,h}(k)$ with $j=6$, where the exponent of $y$ is the number of terms per partition, $h$. 

\section{Closed form expressions for the case of \texorpdfstring{$h=2$}{}}
\label{sec:3}

While all tables exhibit symmetry, only certain cases can be entirely described by readily apparent closed-form expressions. There is no obvious closed-form expression for the coefficients of $r(x)$, though we can find closed-form expressions for the coefficients of $r_h(x)$ when $h=2$. Note that $n=j$. This notation is used to improve readability, as $j$ is too similar to $i$, which we use to denote the imaginary unit. 

\noindent For $k = 1$ and $k=3:$
\begin{equation}
f_{O_j,2}(k)=\frac{1}{8}\bigl((-1)^{1+n} - 1 + (-i)^n + i^n + (-1)^n n + n\bigr)
\end{equation}
For $k=2$:

\noindent We can write the following recurrence relation for this sequence \cite{nine}.
\begin{equation}
    f_{O_j,2}(k)=1+f_{O_{j-2},2}(k)
\end{equation}
\begin{displaymath}
\text{where} \  f_{O_3,2}(k)=1 \ \text{and} \ f_{O_4,2}(k)=2
\end{displaymath}
Solving the recurrence relation, we find
\begin{equation}
    f_{O_j,2}(k)=\frac{1}{4}\bigl((-1)^n + 2n - 1 \bigr)
\end{equation}
For $k>3:$
\begin{equation}
    f_{O_j,2}(k)=0
\end{equation}
We provide a proof for this special case in Section \ref{sec:7}. 

\section{Proof that there exists a solution for all \texorpdfstring{$k$}{}, where \texorpdfstring{$k<j$}{} and \texorpdfstring{$j>2$}{}}
\label{sec:4}

\begin{proposition}
\label{prp:4.1}
Any positive integer $j$ can be written as the sum of a set of fractions, where the numerators are the first $j$ consecutive odd integers and the denominators are $j$. If $j>2$, any positive integer $k$, where $k<j$, can be written as the sum of some combination of the fractions that are summed to produce $j$. 
\end{proposition}

\begin{proof}
We see that pairs of fractions equidistant from opposite ends of the series for $j$ always add to two. Such pairs can be written in the form shown below, where $m=2n-1$:
\begin{equation}
    \frac{m}{j}+\frac{2j-m}{j}=\frac{2j}{j}=2
\end{equation}

\subsection{Odd value for \texorpdfstring{$j$}{}:}
If $j$ is odd, then the series $\sum_{n=1}^{j} \frac{2n-1}{j}$ must have an odd number of terms. The middle term of the series, given by $n=\bigl(\frac{1}{2} (j-1) +1\bigr)$, is always equal to one. 
\begin {equation}
\frac{2\bigl(\frac{(j-1)}{2}+1\bigr)-1}{j}=\frac{j}{j}=1
\end {equation}

Any pair of terms taken from opposite ends of the series $\sum_{n=1}^{j} \frac{2n-1}{j}$ and moving inwards sum to two, with the middle term of the series equaling one. Any even positive integer $k$, where $k<j$, can then be written as some multiple of pairs of fractions that add to two. (For instance, $k=6$ would require three pairs whose terms are equidistant from the ends of the series.) If $k$ is odd, the term from $\sum_{n=1}^{j} \frac{2n-1}{j}$ that equals one would be added to the greatest even integer smaller than $k$ to produce the odd integer $k$. This then proves that any positive integer $k$, where $k<j$, can be written as the sum of some combination of fractions from the series for $j$, where $j$ is odd. 
\subsection{Even value for \texorpdfstring{$j$}{}:}
The property of pairs of terms equidistant from the ends of the series summing to two still holds for even $j$, therefore any even positive integer $k$, where $k<j$, can be written as the sum of some combination of the terms from the series $\sum_{n=1}^{j} \frac{2n-1}{j}$ used to give $j$. 

If $j$ is even and greater than two, the first term $\frac{1}{j}$ and the term given by $n=\frac{j}{2}$ always sum to one. 
\begin {equation}
\frac {1}{j}+\frac{\bigl(2(\frac{j}{2})-1\bigr)}{j}=\frac{j}{j}=1
\end{equation}

If we know the first term and the term given by $n=\frac{j}{2}$ sum to one, we can subtract the pair that adds to one from the total series for $j$, then continue subtracting pairs that add to two until we are left with the odd number of interest. This is possible because $(j-1)$ is odd when $j$ is even, and any odd $(j-1)$ is still odd when a multiple of two is subtracted from it. 

The smallest odd number that can be produced in this way is three, as two fractions that sum to one are removed at the start. This thus removes two possible combinations of terms that would add to two, which would leave the smallest result as four. This is not the case because a pair adding to one was removed from the series, thus leaving $4-1=3$ as the smallest number that can be produced by this method. The proof is still complete; however, as we have shown how every odd $k$ greater than or equal to three can be produced, and have previously shown that the first term $\frac{1}{j}$ and the term given by $n=\frac{j}{2}$ always sum to equal one. This thus proves that any positive odd integer $k$, where $k<j$, can be found as the sum of some combination of fractions from the series for $j$, where $j$ is even. 

Example for $k = 7$ where $j = 10$:
\begin{equation}
    j=10=\sum_{n=1}^{10} \frac{2n-1}{10}=10=\frac{1}{10} + \frac{3}{10} + \frac{5}{10} + \frac{7}{10} + \frac{9}{10} + \frac{11}{10} + \frac{13}{10} + \frac{15}{10} + \frac{17}{10} + \frac{19}{10}
\end{equation}
\begin {equation}
7=\bigl(\sum_{n=1}^{10} \frac{2n-1}{10}\bigr)-\bigl((\frac{1}{10}+\frac{9}{10})+(\frac{3}{10}+\frac{17}{10})\bigr)
=\frac{5}{10}+\frac{7}{10}+\frac{11}{10}+\frac{13}{10}+\frac{15}{10}+\frac{19}{10}
\end{equation}
We see that the value $k=7$ is produced when fractions adding to one and two are subtracted from the series that gives 10. 

Example for $k=3$, where $j=10$
\begin{equation}
\bigl(\sum_{n=1}^{10} \frac{2n-1}{10}\bigr)-\bigl((\frac{1}{10}+\frac{9}{10})+(\frac{3}{10}+\frac{17}{10})+(\frac{5}{10}+\frac{15}{10})+(\frac{7}{10}+\frac{13}{10})\bigr)=\frac{11}{10}+\frac{19}{10}
\end{equation}

Rewriting the above statement, we find:
$10-(1+2+2+2)=\frac{11}{10}+\frac{19}{10}=3$.
These two fractions remain at the end because they are the two fractions that would have summed to two with $\frac{9}{10}$ and $\frac{1}{10}$ respectively, but these were removed to produce a pair adding to one. We thus see that if $j$ is even, three is the smallest odd integer that can be produced by subtracting pairs that add to one and two, as only two fractions remain, which when summed equal three.  

\begin{remark}
We can see that for $j=2$:
\begin{equation}
j=\sum_{n=1}^{2} \frac{2n-1}{2}=\frac{1}{2}+\frac{3}{2}
\end{equation}
\end{remark}

It is not possible to write any integer other than two as the sum of these fractions. (The integer one cannot be found as a sum, as there is only one pair, which when summed gives two. We also note that neither fraction is an integer.) It is only with $j \geq 3$ that we begin to have more than a single pair of fractions in the series, or a term in the series that equals one, thereby enabling one to find combinations for $k<j$. This then implies we must require $j>2$. 
\end{proof}

\section{Proofs relating to unimodality and symmetry}
\label{sec:5}
\subsection{Proofs for relevant lemmas}

Before beginning the proof, we provide necessary background on the Gaussian binomial coefficients. These are q-analogs of the usual binomial coefficients, hence they are polynomials in $q$ that reduce to the usual binomial coefficients when one takes the limit as $q$ approaches one. One can define a Gaussian binomial coefficient as the following, where we note that it is defined as zero when $h>j$ \cite{ten}.
\begin{equation}
\binom{j}{h}_q = \frac{(1-q)(1-q^{2}) \cdots (1-q^{j})}{(1-q)(1-q^2) \cdots (1-q^h)(1-q)(1-q^2) \cdots (1-q^{j-h})}
\end{equation}
These appear in problems such as the counting of lattice paths. In some cases one would wish to set $q$ to a particular value, though its value is irrelevant in our proofs. 
\begin{lemma}
\label{unimodal}
The nonzero terms of $f_{O_j,h}(k)$ are unimodal.
\end{lemma}
\begin{proof}
There is a known generating function for partitions of an integer $m$ into $h$ parts drawn from $\{ 1, 2, \ldots, j\}$, where the exponent of $q$ is $m$. See Section 1.6 of Aigner's textbook for a discussion of this formula \cite{eleven}. The generating function of interest is
\begin{equation}
    q^{\binom{h+1}{2}} \binom{j}{h}_q
\end{equation}
This is a polynomial in $q$ that can be read in the same way as a typical generating function. 

In our problem we partition an integer $kj$ into $h$ terms drawn from the set of the first $j$ odd integers. If one does not require $k$ to be an integer, one can form a bijection by mapping each summand drawn from $\{ 1, 2, \ldots, j\}$ to the corresponding odd integer. We then see that the number of partitions of the integer $\frac{h(h+1)}{2}+c$ into $h$ parts drawn from $\{ 1, 2, \ldots, j\}$ is equal to the number of partitions of $h^2 + 2c$ into $h$ parts from $\{ 1,3,5,\ldots,2j-1 \}$, where $c \in \mathbb{N}$.

Say we have a set $A$ of $h$ terms from the set $\{ 1,2,3,\ldots,j-1,j \}$. Now suppose $A$ has the same sum as another set $B$ of $h$ terms from $\{ 1,2,3,\ldots,j-1,j \}$. If we replace each term of the two sets with the corresponding odd integers, forming the sets $C$ and $D$, then we see that $C$ has the same sum as $D$. This is because we have multiplied each sum by $2$ before subtracting $h$. We then see that the coefficients of the generating function for partitions into $h$ distinct parts from $\{ 1,2,3,\ldots,j-1,j \}$ must be the same as the coefficients of the generating function for partitions into $h$ distinct parts from $\{ 1,3,5,\ldots,2j-1 \}$, as each partition in the former case maps to one partition in the latter case. This indicates that there is a bijection, where the exact rule given in the previous paragraph readily follows from well-known formulae for the sum of the first $h$ integers and the sum of the first $h$ odd integers. 

For example, if $h=3$, we see that $1+4+6=1+2+8=11$. We also see that $1+7+11=1+3+15=19=2(11)-3$, as we have mapped $1$ to the first odd integer, $4$ to the fourth odd integer, $6$ to the sixth odd integer, et cetera. This bijection only works because we have restricted the number of terms to $h$. We cannot do this for $r(x)$, which makes working with it harder. 

To illustrate this bijection, one can examine the following generating functions, which correspond to partitions into $h$ parts drawn from the first $j$ integers and the first $j$ odd integers, respectively. Note that the exponent of $y$ is the number of parts. 

\begin{equation}
    \displaystyle\prod_{n=1}^{j} (1+yx^n)
\end{equation}
\begin{equation}
    \displaystyle\prod_{n=1}^{j} (1+yx^{2n-1})
\end{equation}

If one examines the terms that share a given exponent $h$ of $y$, one finds that the two polynomials have the same coefficients. All that changes are the exponents of $x$. For example, we have:
\begin{equation}
    \displaystyle\prod_{n=1}^{6} (1+yx^n) = x^{21}y^6+x^{20}y^5+x^{19}y^5+x^{18}y^5+x^{18}y^4+x^{17}y^5+x^{17}y^4+x^{16}y^5+2x^{16}y^4
\end{equation}
\begin{equation*}
    +x^{15}y^5+2x^{15}y^4+x^{15}y^3+3x^{14}y^4+x^{14}y^3+2x^{13}y^4+2x^{13}y^3+2x^{12}y^4+3x^{12}y^3+x^{11}y^4+3x^{11}y^3
\end{equation*}
\begin{equation*}
    +x^{11}y^2+x^{10}y^4+3x^{10}y^3+x^{10}y^2+3x^9y^3+2x^9y^2+2x^8y^3+2x^8y^2+x^7y^3+ 3x^7y^2+x^6y^3+2x^6y^2
\end{equation*}
\begin{equation*}
    +x^6y+2x^5y^2+x^5y+x^4y^2+x^4y+x^3y^2+x^3y+x^2y+xy+1
\end{equation*}

When we pull out the terms containing $y^2$, hence corresponding to two terms per partition, we have
\begin{equation}
    x^{11}+x^{10}+2x^9+2x^8+3x^7+2x^6+2x^5+x^4+x^3
\end{equation}
Repeating the process with the odd case:
\begin{equation}
    \displaystyle\prod_{n=1}^{6} (1+yx^{2n-1}) = x^{36}y^6+x^{35}y^5+x^{33}y^5+x^{32}y^4+x^{31}y^5+x^{30}y^4+x^{29}y^5+2x^{28}y^4+x^{27}y^5
\end{equation}
\begin{equation*}
    x^{27}y^3+2x^{26}y^4+x^{25}y^5+x^{25}y^3+3x^{24}y^4+2x^{23}y^3+2x^{22}y^4+3x^{21}y^3+2x^{20}y^4+x^{20}y^2+3x^{19}y^3
\end{equation*}
\begin{equation*}
    +x^{18}y^4+x^{18}y^2+3x^{17}y^3+x^{16}y^4+2x^{16}y^2+3x^{15}y^3+2x^{14}y^2+2x^{13}y^3+3x^{12}y^2+x^{11}y^3+x^{11}y
\end{equation*}
\begin{equation*}
    +2x^{10}y^2+x^9y^3+x^9y+2x^8y^2+x^7y+x^6y^2+x^5y+x^4y^2+x^3y+xy+1
\end{equation*}
Pulling out terms containing $y^2$:
\begin{equation}
    x^{20}+x^{18}+2x^{16}+2x^{14}+3x^{12}+2x^{10}+2x^8+x^6+x^4
\end{equation}
Both sequences of coefficients follow those of 
\begin{equation}
q^{\binom{2+1}{2}} \binom{6}{2}_q = q^{11}+q^{10}+2q^9+2q^8+3q^7+2q^6+2q^5+q^4+q^3
\end{equation}

When we require $k$ to be an integer, we require each exponent of $x$ to be an integral multiple of $6$, as $j=6$. After pulling those terms out from \begin{equation}
    x^{20}+x^{18}+2x^{16}+2x^{14}+3x^{12}+2x^{10}+2x^8+x^6+x^4
\end{equation}
We have
\begin{equation}
    x^{18}+3x^{12}+x^6
\end{equation}
As expected, this is unimodal. 

It is known that the Gaussian binomial coefficients are unimodal. A constructive combinatorial proof was produced by O'Hara \cite{eleven}. If one knows that $\binom{j}{h}_q$ is unimodal, one sees that $q^{\binom{h+1}{2}} \binom{j}{h}_q$ is unimodal as well, as multiplying each term by the same power of $q$ does not affect unimodality. As we have a bijection to our problem when $k$ is not necessarily an integer, we see we have unimodality if $k$ is not required to be an integer. This generating function corresponds to the terms of the following polynomial, where we only examine terms with the same exponent of $y$. 
\begin{equation}
    \displaystyle\prod_{n=1}^{j} (1+yx^{2n-1})
\end{equation}

The nonzero terms of $r_h(x)$ are the terms from the polynomial above where the exponent of $x$ is $kj$, $k \in \mathbb{N}$. If one does not change the order of the terms relative to one another, any polynomial composed of terms taken from a unimodal polynomial must also be unimodal. We thus see that the nonzero coefficients of $r_h(x)$ are unimodal. 
\end{proof}

\begin{lemma}
\label{symmetric}
The coefficients of $r_h(x)$ are symmetric about the coefficient corresponding to $k=h$.
\end{lemma}

\begin{proof}
Picture a group of $h$ terms that sum to some value $k_1$. If we measure the distance of each term from the left end of the series for $j$, then replace that term with a term located that many units from the right end of the series, we find a group of terms that sum to $(2h-k_1)$, hence corresponding to $k=2h-k_1$. We observe that the values $k_1$ and $2h-k_1$ are equidistant from $k=h$. 
\begin{equation*}
    \frac{2(n_1)-1}{j}+\frac{2(n_2)-1}{j}+\cdots+\frac{2(n_{h-1})-1}{j}+\frac{2(n_h)-1}{j}=\frac{2(n_1+n_2+\cdots+n_h)-h}{j}=k_1
\end{equation*}
\begin{equation*}
    \frac{2(j-n_1+1)-1}{j}+\frac{2(j-n_2+1)-1}{j}+\cdots+\frac{2(j-n_{h-1}+1)-1}{j}+\frac{2(j-n_h+1)-1}{j}
\end{equation*}
\begin{equation}
    =\frac{2hj+h-2(n_1+n_2+\cdots+n_{h-1}+n_h)}{j}= 2h - k_1
\end{equation}

We observe that we are really reflecting the group of selected boxes shown below over a vertical line located in the center of the series. Any combination for $k_1$ can then be reflected about a central vertical line to find a new combination for $2h-k_1$. As $k_1$ and $2h-k_1$ are equidistant from $k=h$, we can then see that we have symmetry about $k=h$. We write an example below, where the odd numerators are listed out. Boxes are drawn around terms used in a sum. Here $k=5$, $j=8$, and $h=4$:
\begin{center}
1 \fbox{3} 5 7 \fbox{9} 11 \fbox{13} \fbox{15}
\end{center}
We see that $3+9+13+15=40=(5)(8)$. After reflecting over a central axis, we have:
\begin{center}
\fbox{1} \fbox{3} 5 \fbox{7} 9 11 \fbox{13} 15
\end{center}
Note that $1+3+7+13=24=(3)(8)$, hence corresponding to $k=3$. 

As we pointed out in the proof for unimodality that this problem is connected to the Gaussian binomial coefficients, it appears that this proof for symmetry can be combined with the bijection in the proof for unimodality to potentially serve as an excessively long alternative proof that $\binom{j}{h}_q$ is symmetric. The symmetry of Gaussian binomial coefficients is a well-known result, though it is usually proven in a different, more direct way. 
\end{proof}
\begin{lemma}
\label{Sum_to_h}
A group of $h$ consecutive terms from the center of the series for $j$ sums to $k=h$. The first $\frac{h}{2}$ consecutive terms and the last $\frac{h}{2}$ consecutive terms together sum to $k=h$.
\end{lemma}
\begin{proof}
\begin{equation*}
    \frac{2(\frac{j-h}{2}+1)-1}{j}+\frac{2(\frac{j-h}{2}+2)-1}{j}+\cdots+\frac{2(\frac{j-h}{2}+h-1)-1}{j}+\frac{2(\frac{j-h}{2}+h)-1}{j}
\end{equation*}
\begin{equation}
    =\frac{h(j-h)+h(h+1)-h}{j}=\frac{hj}{j}=h
\end{equation}
We thus see that a group of $h$ terms centered in the series for $j$ sums to $k=h$.

\begin{equation*}
    \frac{2(1)-1}{j}+\frac{2(2)-1}{j}+\cdots+\frac{2(\frac{h}{2})-1}{j}+ \frac{2(j)-1}{j}+\frac{2(j-1)-1}{j}+\cdots+\frac{2j+1-h}{j}
\end{equation*}
\begin{equation*}
    =\frac{\frac{h}{2}(\frac{h}{2}+1)-\frac{h}{2}}{j}+\frac{2(j+(j-1)+(j-2)+\cdots+(j-\frac{h}{2}+1))-\frac{h}{2}}{j}
\end{equation*}
\begin{equation}
    =\frac{\frac{h}{2}(\frac{h}{2}+1)-\frac{h}{2}}{j}+\frac{hj-(\frac{h}{2}-1)\frac{h}{2}-\frac{h}{2}}{j}=\frac{hj}{j}=h
\end{equation}
This then indicates that the first $\frac{h}{2}$ terms and the last $\frac{h}{2}$ terms together sum to $k=h$.
\end{proof}

\subsection{Proof for Theorem \ref{thm1}}
\begin{proof}
By Lemmas \ref{unimodal} and \ref{symmetric} the nonzero coefficients of $r_h(x)$ are unimodal and symmetric about the coefficient corresponding to $k=h$. These results alone do not prove that the maximum nonzero coefficient of $r_h(x)$ is $f_{O_j,h}(h)$, though, as one must dismiss the scenario where $f_{O_j,h}(h)=0$ but $f_{O_j,h}(k)$ is nonzero for values of $k$ equidistant from $h$. 

We see that there is at least one partition for $k=h$ when both $h$ and $j$ are even, or both $h$ and $j$ are odd, as in these cases it is possible to find a centered group of $h$ consecutive terms from the series for $j$, which by Lemma \ref{Sum_to_h} sum to $h$. If $h$ is even and $j$ is odd, by Lemma \ref{Sum_to_h} one can take the first $\frac{h}{2}$ terms and the last $\frac{h}{2}$ terms from the series for $j$, which also sum to $k=h$. If $h$ is odd and $j$ is even, one cannot find a centered group of $h$ terms or two groups of $\frac{h}{2}$ terms. This is not a problem, though, as no $k$ can be found as the sum of an odd number of terms when $j$ is even. The sum of an odd number of odd integers is odd, though $j$ is even, hence an odd sum divided by an even $j$ cannot give an integral $k$. We thus have symmetry about $k=h$, see that $f_{O_j,h}(k)$ is unimodal, and know that we have at least one partition for $k=h$ when there are partitions for any integer $k$, where $0<k<j$. This then indicates that $k=h$ must have the most partitions with a given $j$, or none at all, in which case every value of $f_{O_j,h}(k)$ is zero with that given $j$. 
\end{proof}

\subsection{Proof for proposition \ref{prop1}}
\begin{proof}
Suppose we consider two $k$-values of $k_1$ and $(j-k_1)$, where we note that $k_1$ and $(j-k_1)$ are equidistant from the ends of the sequence $\{ 1,2,3,\ldots,j-1 \}$. We know that each combination for $k_1$ is a portion of the larger series for $j$, thus any combination for the $k$-value of $(j-k_1)$ is what is left over whenever each possible combination for $k_1$ is removed from the larger series for $j$. This then necessitates that $k_1$ and $(j-k_1)$ must have the same number of combinations, as if there are $M$ combinations for $k_1$, then there are only $M$ ways to subtract $k_1$ from $j$, thus causing there to also be $M$ combinations for $(j-k_1)$. We then have $f_{O_j}(k_1)=f_{O_j}(j-k_1) \forall k_1<j$.
\end{proof}

\section{Even numerators}
\label{sec:6}
\subsection{Proof for Theorem \ref{thm2}}
We note that the sum of the first $x$ even integers can be given by 

\begin{equation}
    \sum_{n=1}^{x} 2n = x^2+x
\end{equation}

We can thus transform our problem into counting solutions for a given $t$ to the following Diophantine equation, where $x,y,t \in \mathbb{N}$ and $0<y<x<t$. Note that the restriction $0<y<x<t$ serves to prevent one from being able to find obvious solutions that hold for all $t$.

\begin{equation}
    \frac{x^2+x}{y}=t
\end{equation}
For matters of convenience that will later become clear, it is useful to alter our restrictions within the majority of our proofs, temporarily inserting two trivial solutions. In fact, if we alter our restrictions to $0<x \leq t$ and $0<y \leq t+1$, we only add the trivial solutions of $x=t$ and $x=t-1$, which are solutions for every $t$. 

\newpage

\begin{lemma}
There exist exactly two integral solutions to $t=\frac{x^2+x}{y}$ that satisfy $0<x\leq t$ and $0<y\leq (t+1)$ but not $0<y<x<t$.
\end{lemma}

\begin{proof}
Let $y>(x+1)$, thus $\frac{x+1}{y}<1$. We then see $\frac{x(x+1)}{y}=t<x$, which is not allowed by our new restriction $0<x\leq t$. We see that $y$ can only equal $x+1$ when $t=x$, which is the trivial solution of $y=t+1$. The variable $y$ can only equal $x$ when $t$ is $x+1$, which is the trivial solution $y=t-1$. $y$ cannot equal $t$, as this would imply that $t^2 = x(x+1)$. For a natural number to be a perfect square, its prime factors must all appear an even number of times. Consecutive natural numbers share no prime factors, thus both $x$ and $x+1$ would have to be perfect squares, which is impossible. We then see that all remaining solutions fit the restriction $0<y<x<t$, hence the two trivial solutions are the only points added by the restriction change.  
\end{proof}

\begin{lemma}
The number of solutions for a given $t$ is equal to the product of the number of solutions for each $p_{i}^{j}$, where we have the prime factorization $t=p_{i_1}^{j_1} p_{i_2}^{j_2} p_{i_3}^{j_3} \cdots$.
\end{lemma}

\begin{proof}
We see that we can rewrite our problem as 
\begin{equation}
x^2+x \equiv 0\bmod t
\end{equation}
It is known that if one has a polynomial $f(x)$, where 
\begin{equation}
f(x) \equiv 0\bmod t
\end{equation}
the number of solutions for a given $t$ can be given by the product of the number of solutions to each
\begin{equation}
    f(x) \equiv 0\bmod p_i^j
\end{equation}
where we have the prime factorization $t=p_{i_1}^{j_1} p_{i_2}^{j_2} p_{i_3}^{j_3} \cdots$. This is a consequence of the ring isomorphism of the Chinese Remainder Theorem.  
\end{proof}

\begin{lemma}
Let $x$ be a positive integer, where $x \neq t-1$ and $x$ is a solution to $\frac{x^2+x}{y}=t$, subject to the restrictions $0<x \leq t$ and $0<y \leq (t+1)$. $x$ must then be a multiple of at least one prime factor of $t$.
\end{lemma}

\begin{proof}
We observe that $x$ and $x+1$ are relatively prime, thus they share no prime factors. If $x$ and $t$ share no prime factors, $y$ must cancel all prime factors of $x$. We then have $y=Ax$, where $A$ is a natural number. If $A=1$, $x=t-1$, which is a trivial solution. 

Let $A>1$:
\begin{equation}
    t=\frac{x(x+1)}{Ax}=\frac{x+1}{A}
\end{equation}
We thus see that $t<x$, which is not allowed by our restriction $0<x\leq t$. This then indicates that the only way for $t$ and $x$ to share no prime factors is if $x=t-1$, hence indicating that all other $x$ one can find as solutions must share at least one prime factor with $t$.
\end{proof}

\newpage

\begin{lemma}
If $t$ is prime, then there are only two solutions. 
\end{lemma}
\begin{proof}
We know that $x$ and $t$ must share at least one prime factor, or none at all, where if they share none, then $x=t-1$. If $t$ is prime, it only has one prime factor, so $x$ must be $t-1$ or an integral multiple of $t$. Let $a \in \mathbb{Z}$, $a>0$, where $x=at$. If $a=1$, we have $x=t$ and $y=t+1$, which is one of our two trivial solutions. 
\begin{equation}
    \frac{x^2+x}{y}=\frac{t^2+t}{t+1}=t
\end{equation}
If $a>1$, we see that $y$ must increase, though that would mean $y>t+1$, which our restrictions $0<x \leq t$ and $0<y \leq t+1$ forbid. We thus have only two solutions when $t$ is prime, where these two solutions are the trivial solutions. 
\end{proof}
\begin{lemma}
Let $p$ be a prime number, and $n$ a natural number. The number of solutions for $p$ is the same as the number of solutions for $p^n$. 
\end{lemma}
\begin{proof}
We can split $x^2+x$ into its factors, hence enabling us to look at the following two subproblems:
\begin{equation}
    g(x)=x\equiv 0\bmod p
\end{equation}
\begin{equation}
    h(x)=x+1\equiv 0\bmod p
\end{equation}
The integer one is not a prime, thus we see that 
\begin{equation}
    \frac{dg}{dx}=1\not\equiv 0\bmod p
\end{equation}
\begin{equation}
    \frac{dh}{dx}=1\not\equiv 0\bmod p
\end{equation}
By Hensel's Lemma, there are unique solutions to $x\equiv 0\bmod p$ and $x+1\equiv 0\bmod p$. The polynomial $x^2+x$ has no repeated factors. We also note that $x$ and $x+1$ are relatively prime, hence this indicates that there are only two solutions for $x^2+x\equiv 0\bmod p^n$, where we see those solutions are $x=p^n$ and $x=p^n - 1$. In other words, we have the same number of solutions for $t=p$ and $t=p^n$.
\end{proof}
We are now in a position to combine our lemmas to prove the theorem. 

\begin{proof}
We know that the number of solutions to 
\begin{equation}
x^2+x \equiv 0\bmod t
\end{equation}
is equal to the product of the number of solutions to each
\begin{equation}
    x^2+x \equiv 0\bmod p_i^j
\end{equation}
where the prime factorization of $t$ is $t=p_{i_1}^{j_1} p_{i_2}^{j_2} p_{i_3}^{j_3} \cdots$. As there are two solutions for $t=p^n$, we see that there are $2^{\omega(t)}$ solutions for 
\begin{equation}
    \frac{x^2+x}{y}=t
\end{equation}
where $0<x\leq t$ and $0<y\leq t+1$, $x,y,t \in \mathbb{N}$. As this restriction change adds two trivial solutions, we thus see that for the restriction $0<y<x<t$, we have $2^{\omega(t)}-2$ solutions, hence $F_E(t)=2^{\omega(t)} - 2$. 
\end{proof}

\subsection{Proof for corollary \ref{cor2}}
\begin{proof}
In this section we provide a proof for the continuation we give in Corollary \ref{cor2}. We apply a slightly different take on the circle method and the concept of a generating function to achieve this. In the following subsection, we use part of this process as a stepping stone to produce the series identity in Corollary \ref{cor3}. If we examine our Diophantine equation, we see that if we have found a solution, we have 

\begin{equation}
    \bigg(\frac{x^2+x}{y}-t \bigg) = 0
\end{equation}

Shifting our restrictions from $0<y<x<t$ to $0<x\leq t$ and $0<y\leq t+1$ is convenient in part because it removes the need for $y$ to always be less than $x$. We can then eliminate $y$ by taking a product over its range of values. We see that whenever we find a solution, one of the factors equals zero, thus causing the entire polynomial to equal zero. 

\begin{equation}
    \displaystyle\prod_{y=1}^{t+1} \bigg(\frac{x^2+x}{y}-t \bigg) = 0
\end{equation}
After multiplying both sides by $(t+1)!$, we can simplify to 
\begin{equation}
    \Psi_t=\displaystyle\prod_{y=1}^{t+1}(x^2+x-yt)=0
\end{equation}

This is a family of polynomials in $x$, as $y$ is an index within the product and $t$ becomes some positive index that is constant within this product. We thus have a different polynomial $\Psi_t$ for each value of $t$. In a partition problem, one would normally work with the coefficients of the generating function, which is a polynomial, though instead we accomplish the same sort of thing by counting the positive integral roots of a different polynomial, which we call $\Psi_t$. This is a bit more natural for our partition problem than a typical, non-Dirichlet generating function. For comparison, we write the generating function for this problem below.

\begin{equation}
    F_E(t_1)x^{t_1}+F_E(t_1)x^{t_2}+F_E(t_3)x^{t_3}+F_E(t_4)x^{t_4}+\cdots
\end{equation}

We have a different function $\Psi_t$ associated with each $t$, where $\Psi_t$ is a function of $x$ alone, hence the number of positive integral roots of $\Psi_t$ is a function of $t$. We count the number of positive integral roots with the function $R(t)$. As we have added two trivial solutions, we see that $R(t)=F_E(t)+2$. Each coefficient in the generating function is then two less than the number of positive integral roots of $\Psi_t$, where one uses the appropriate value of $t$. For example, when $t=6$, we have
\begin{equation}
\Psi_6=x^{14}+7x^{13}-147x^{12}-973x^{11}+9107x^{10}+54621x^9-309953x^8
\end{equation}
\begin{equation*}
    -1578527x^7+6290256x^6+24636024x^5-76219920x^4-195456240x^3
\end{equation*}
\begin{equation*}
    +507586176x^2+609700608x-1410877440
\end{equation*}

We see that the positive integral roots are $x=2$, $x=3$, $x=5$, and $x=6$. These are the values of $x$ that are solutions when $t=6$, where we include the trivial solutions $x=5$ and $x=6$. 

By Theorem \ref{thm2}, we know that $R(t)=2^{\omega(t)}$. We now examine the following integral. We place $\Psi_t$ in the exponent of $z$, where the coefficient of $z$ is one. This may seem odd, though it becomes convenient later. Note that $\gamma$ is the unit circle in the complex plane, oriented counterclockwise. 

\begin{equation}
    y,t,x \in \mathbb{Z} \therefore \frac{1}{2\pi i}\int_\gamma z^{\Psi_t-1} dz \in \{0,1\}
\end{equation}

We have a solution when $\Psi_t=0$. This leaves the exponent of $z$ as negative one, hence causing the integral to evaluate as one. As the coefficient of $z$ is always exactly one, we have no convergence issues from the integral itself, as it is always identically zero or one. Rather than divide the circle into major and minor arcs, we instead repeatedly integrate over the undivided circle. As mentioned previously, we refer to the number of roots of $\Psi_t$ as $R(t)$. 

\begin{equation}
    R(t)=\sum_{x=1}^{t} \frac{1}{2\pi i}\int_\gamma z^{\Psi_t-1} dz
\end{equation}

We observe that we can rewrite the integral, where $\sinc$ is the normalized $\sinc$ function:
\begin{equation}
    \sum_{x=1}^{t} \frac{1}{2\pi i}\int_\gamma z^{\Psi_t-1} dz = \sum_{x=1}^{t} \sinc(\Psi_t)
\end{equation}
This works because $\sinc(m)=1$ when $m=0$, and zero when $m$ is a nonzero integer. We observe that since $\Psi_t$ is really a product of polynomials, where $x$, $y$, and $t$ are integers, $\Psi_t$ must output an integer for every value in our problem.  

\begin{figure}[H]
  \caption{Plot of the normalized function $\sinc(x)$}
  \centering
    \includegraphics[height = 5 cm]{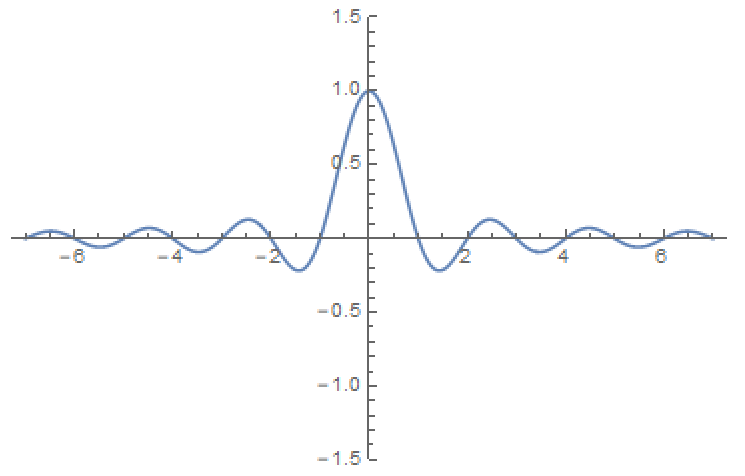}
\end{figure}
We thus have:
\begin{equation}
    2^{\omega(t)}=\bigg( \sum_{x=1}^{t} \sinc \bigg( \displaystyle\prod_{y=1}^{t+1}(x^2+x-yt) \bigg) \bigg)
\end{equation}

To extend the definition of this function to non-integral inputs and inputs with an imaginary part, we alter the definition, where we see that the expression still correctly counts the number of distinct prime factors of the positive integers. We rename $t$ to $z$ to highlight that in general, the input of this function is complex. 
\begin{equation}
    \omega(z)=\log_2 \bigg( \sum_{x=1}^{\left \lceil{\Re(z)}\right \rceil } \sinc \bigg( \displaystyle\prod_{y=1}^{\left \lceil{\Re(z)}\right \rceil+1}(x^2+x-yz) \bigg) \bigg)
\end{equation}

We thus have a continuation of the prime omega function, though we note that the inclusion of the ceiling function prevents it from being analytic everywhere. This enables us to define $\omega(z)$ for inputs that are not positive integers. This is somewhat similar to the practice of using the gamma function to assign a value to expressions such as the factorial of $\frac{1}{2}$, which otherwise would be absurd. For example, we can \enquote{count} the number of distinct prime factors of numbers such as $\pi$ and $e$. Such numbers obviously have no prime factors, though it is interesting to see that we can assign them a quantity of such factors by extending $\omega(z)$ to the complex plane. 
\begin{equation}
    \omega(\pi) \approx -9.9287 
\end{equation}
\begin{equation}
    \omega(e) \approx -6.0963 + 4.5323i
\end{equation}
\begin{equation}
    \omega(4+i) \approx 181729.6967 - 0.0798i
\end{equation}
\end{proof}

\subsection{Proof for corollary \ref{cor3}}
\begin{proof}
We observe the following, where the expression $\bigg(1-\frac{f(x)}{t} \bigg)_{t+1}$ is a pochhammer symbol. 
\begin{equation}
    \displaystyle\prod_{n=1}^{n+1}(f(x)-nt) = (-t)^{t+1} \bigg(1-\frac{f(x)}{t} \bigg)_{t+1}
\end{equation}
\begin{equation}
    (x)_n = \frac{\Gamma(x+n)}{\Gamma(x)}
\end{equation}
Hence when $x$ and $t$ are integers, we have
\begin{equation}
    \Psi_t=\displaystyle\prod_{y=1}^{t+1}(x^2+x-yt)=\frac{(-t)^{t+1}\Gamma(\frac{-x^2}{t}-\frac{x}{t}+t+2)}{\Gamma(\frac{-x^2}{t}-\frac{x}{t}+1)}
\end{equation}
We now write a Dirichlet series $D(s)$:
\begin{equation}
    D(s)=\sum_{t=1}^{\infty} \frac{a_t}{t^s}
\end{equation}
\begin{equation}
    a_t = \sum_{x=1}^{t} (\sinc(\Psi_t))
\end{equation}

We specify that $s$ is a real number greater than two. Since real numbers are complex numbers with an imaginary part of zero, $s$ is still complex. 
\begin{equation}
     D(s)=\sum_{t=1}^{\infty} \sum_{x=1}^{t} \frac{1}{2\pi i}\int_\gamma \frac{z^{\Psi_t-1}}{t^s} dz = \sum_{t=1}^{\infty} \sum_{x=1}^{t} \frac{\sinc(\Psi_t)}{t^s} 
\end{equation}

The coefficient $a_t$ increases by one every time we find a positive integral root of $\Psi_t$ using that value of $t$. $\Psi_t$ is a polynomial of degree $2t+2$, thus it cannot have more than $2t+2$ such roots for a given $t$. In our case, we cannot have more than one solution for each value of $x$ with a given $t$, where $0<x\leq t$, hence we have at most $t$ solutions. We thus have the following comparison:
 \begin{equation}
     \sum_{t=1}^{\infty} \sum_{x=1}^{t} \frac{\sinc(\Psi_t)}{t^s}   \leq \sum_{t=1}^{\infty} \frac{t}{t^s}
 \end{equation}
 As $s$ is a real number greater than two, we see that the series on the right is a convergent p-series.
 \begin{equation}
     \sum_{t=1}^{\infty} \frac{t}{t^s}= \sum_{t=1}^{\infty} \frac{1}{t^{s-1}} 
 \end{equation}
As $D(s)$ is always less than or equal to a convergent series, it must also converge. 

\noindent The following Dirichlet series is known \cite{twelve}.
\begin{equation}
    \sum_{n=1}^{\infty} \frac{2^{\omega(n)}}{n^s} = \frac{\zeta^2(s)}{\zeta(2s)}
\end{equation}
We thus have
\begin{equation}
\sum_{t=1}^{\infty} \sum_{x=1}^{t} \frac{\sinc(\Psi_t)}{t^s}  = \frac{\zeta^2(s)}{\zeta(2s)}
\end{equation}
Hence
\begin{equation}
\sum_{t=1}^{\infty} \sum_{x=1}^{t} \frac{\sinc\bigg(\frac{(-t)^{t+1}\Gamma(\frac{-x^2}{t}-\frac{x}{t}+t+2)}{\Gamma(\frac{-x^2}{t}-\frac{x}{t}+1)} \bigg)}{t^s}  = \frac{\zeta^2(s)}{\zeta(2s)}
\end{equation}
\end{proof}

\newpage 

\section{Proof for the special case of restriction to two terms}
\label{sec:7}

We divide this proof into sections, focusing on each value of $k$ in turn. We prove that the terms of the table corresponding to $h=2$ (given on page 4) must always follow certain simple patterns, which then lead to the closed-form expressions given in Section \ref{sec:3}.

\begin{lemma}
\label{k1}
When $k=1$, we have $f_{O_j,2}(k)=\frac{1}{8}\bigl((-1)^{1+n} - 1 + (-i)^n + i^n + (-1)^n n + n\bigr)$, where $n=j$.
\end{lemma}

\begin{proof}
An odd value of $j$ cannot have any combination of two terms that add to $k=1$, as the sum of any two odd numbers is even, and thus unequal to the odd denominator. (Any pairs of numerators chosen must sum to the denominator to produce $k=1$.) This then results in the alternating zeros found with the column for $k=1$. We know that any even value of $j$ has at least one pair of terms that sum to $k=1$, as we proved in Section \ref{sec:4} that the terms given by $n=1$ and $n=\frac{j}{2}$ sum to this value. If two is subtracted from one of the resulting numerators, then added to the other, the sum does not change. The numerators in the series for $j$ are consecutive odd integers, thus they differ from their neighbors by two. We thus find that the number of pairs that can sum to one for an even $j$ is equal to the number of pairs equidistant from $n=1$ and $n=\frac{j}{2}$, as we can move inwards from those terms to produce more combinations that work. This results in the pattern $\{ 0,1,0,1,0,2,0,2,0,3,0,\ldots \}$ seen for $k=1$. 

\noindent Example for $j=8$:
\begin{equation}
8=\frac{1}{8}+\frac{3}{8}+\frac{5}{8}+\frac{7}{8}+\frac{9}{8}+\frac{11}{8}+\frac{13}{8}+\frac{15}{8}
\end{equation}
Pulling out the section of the series from $n=1$ to $n=\frac{j}{2}$ we have:
\begin{equation}
\frac{1}{8}+\frac{3}{8}+\frac{5}{8}+\frac{7}{8}
\end{equation}

We see that there are two pairs equidistant from the ends of this section of the series. Each of these sum to one. 
\begin{equation}
    1=\frac{1}{8}+\frac{7}{8}  
\end{equation}
\begin{equation}
1=\frac{3}{8}+\frac{5}{8}
\end{equation}

The numerator $3=1+2$, and the numerator $5=7-2$. The number of these pairs increases by one for every other even $j$, therefore we get the pattern seen in the table. Through an argument nearly identical to that used to prove Lemma \ref{TwoTerms}, one can see that only pairs equidistant from the terms given by $n=1$ and $n=\frac{j}{2}$ can sum to one.

We thus know that for $k=1$, $f_{O_j,2}(k)$ must always follow the simple pattern seen in the first column of the table. We can see that this pattern is generated by the formula $\frac{1}{8}\bigl((-1)^{1+n} - 1 + (-i)^n + i^n + (-1)^n n + n\bigr)$, hence $f_{O_j,2}(k)$ has this closed-form expression when $k=1$.
\end{proof}

\newpage

\begin{lemma}
\label{TwoTerms}
Only pairs of terms equidistant from the ends of the series for $j$ can sum to two.
\end{lemma}

\begin{proof}
Suppose there are two positive odd integers, $a$ and $b$, where 
\begin{equation}
a+b=kj 
\end{equation}
\begin{equation}
a<b\leq(2j-1)
\end{equation}

This thus requires $a$ and $b$ to be numerators within the series for $j$, where the fractions in question sum to $k$, as the denominators are $j$. 
Suppose there are two other positive odd integers, $c$ and $d$, where:
\begin{equation}
c+d=kj 
\end{equation}
\begin{equation}
c<d<(2j-1)
\end{equation}
For $c+d$ to equal $a+b$, there must exist an integer $Q$ such that: 
\begin{equation}
c=a+Q  
\end{equation}
\begin{equation}
d=b-Q
\end{equation}

If we begin with a pair of integers, $a$ and $b$, where $a$ and $b$ are numerators equidistant from the ends of the series for $j$, we know they sum to $2j$, as we proved in Section \ref{sec:4} that the respective fractions these numerators are a part of sum to two. Any other pair, $c$ and $d$, where $c=a+Q$ and $d=b-Q$ must also be equidistant from the ends of the series for $j$. We then know that such pairs are the only pairs that can sum to $kj$, where $k=2$. 
\end{proof}

\begin{lemma}
\label{k2}
When $k=2$ we have $f_{O_j,2}(k)=\frac{1}{4}\bigl((-1)^n + 2n - 1 \bigr)$, where $n=j$.
\end{lemma}

\begin{proof}

We proved in Section \ref{sec:4} that pairs of terms equidistant from the ends of the series for $j$ always sum to two. By Lemma \ref{TwoTerms}, we also know that such pairs are the only pairs that can sum to two. When $j$ is even, there are $\frac{j}{2}$ pairs equidistant from the ends of the series for $j$, though when $j$ is odd we can find $\frac{(j-1)}{2}$ equidistant pairs. We then see that an even $j$ and the following odd $j$ have the same number of pairs that sum to two. This produces the pattern seen where each number of combinations is displayed twice. 

We thus know that for $k=2$, $f_{O_j,2}(k)$ always follows the pattern $\{ 1,2,2,3,3,4,4,\ldots \}$ seen in the second column of the table. We can see that this pattern is generated by the formula $f_{O_j,2}(k)=\frac{1}{4}\bigl((-1)^n + 2n - 1 \bigr)$, hence $f_{O_j,2}(k)$ has this closed-form expression when $k=2$.
\end{proof}

\begin{lemma}
\label{k3}
When $k=3$, we have $f_{O_j,2}(k)=\frac{1}{8}\bigl((-1)^{1+n} - 1 + (-i)^n + i^n + (-1)^n n + n\bigr)$, where $n=j$.
\end{lemma}

\begin{proof}

By Lemma \ref{symmetric}, symmetry implies that $f_{O_j,2}(1)=f_{O_j,2}(3)$.
\end{proof}

\newpage

\begin{lemma}
\label{kgreaterthan3}
If $k>3$, $f_{O_j,2}(k)=0$.
\end{lemma}

\begin{proof}

Suppose it is possible to find $h$ terms from the series for $j$ that sum to a $k$-value greater than or equal to $2h$, where $h>1$. This then implies the following statement, where $q \geq 0$. 
\begin{equation}
\left( \frac{2n_1 - 1}{j} + \frac{2n_2 - 1}{j} + \cdots + \frac{2n_{h-1} - 1}{j} + \frac{2n_h - 1}{j} \right)=(2h+q)
\end{equation}
\begin{equation}
    2\bigl(n_1 + n_2 + \cdots + n_{h-1} + n_h\bigr)=(2h+q)j+h
\end{equation}
We know that $n_1$, $n_2$, \ldots , $n_{h-1}$, $n_h$ are all positive integers less than or equal to $j$, where at most one is equal to $j$, thus:
\begin{equation}
n_1=j-a_1, \ n_2=j-a_2, \ \cdots n_{h-1}=j-a_{h-1}, \ n_h=j-a_h
\end{equation}

The variables $a_1$, $a_2$, \ldots, \ $a_{h-1}$, $a_h$ are positive integers with at most one equaling zero.
\begin{center}
$a_1, a_2, \cdots ,a_{h-1},$ and  $a_h$ are distinct.
\end{center}
These constants are distinct to prevent any term from being repeated. No term can be repeated because no term in the series for $j$ is repeated, thus no portion of the larger series can contain repeated terms. 
Summing the $n$-values, we then find
\begin{equation}
    hj-\bigl(a_1 + a_2 + \cdots + a_{h-1} + a_h\bigr)=\left( h+\frac{q}{2} \right)j+\frac{1}{2}h
\end{equation}
\begin{equation}
h>1 \ \therefore \ \bigl(a_1 + a_2 + \cdots + a_{h-1} + a_h\bigr)>0
\end{equation}
\begin{equation}
    hj-\bigl(a_1 + a_2 + \cdots + a_{h-1} + a_h\bigr)<hj
\end{equation}
\begin{equation}
    \left( \left(h+\frac{q}{2} \right)j + \frac{1}{2}h \right)>hj
\end{equation}
\begin{equation}
    hj-\bigl(a_1 + a_2 + \cdots + a_{h-1} + a_h\bigr) \neq \left( \left(h+\frac{q}{2} \right)j + \frac{1}{2}h \right)
\end{equation}
\begin{equation}
\left(\frac{2n_1 - 1}{j} + \frac{2n_2 - 1}{j} + \cdots + \frac{2n_{h-1} - 1}{j} + \frac{2n_h - 1}{j} \right) \neq (2h+q)
\end{equation}

We can conclude that if combinations are restricted to $h$ terms, where $h>1$, then it is impossible to find any combination for a value of $k$ greater than or equal to $2h$. If $k>3$ where $h=2$, then $k \geq 2h$, thus there cannot be any combination of $2$ terms for $k > 3$. We have thus proven that the case of $h=2$ must always follow the simple patterns visible in its number table. These patterns can then be produced by the closed-form expressions we give.
\end{proof}

We now see that we can combine these lemmas to prove our result.

\begin{proof}
Combining Lemma \ref{k1}, Lemma \ref{TwoTerms}, Lemma \ref{k2}, Lemma \ref{k3}, and Lemma \ref{kgreaterthan3} we see that $f_{O_j,2}(k)$ follows the closed-form expressions given in Section \ref{sec:3}.
\end{proof}

\section{More examples of number tables}
\label{sec:8}

To generate the number table for $r(x)$, the MATLAB script prompts the user for the maximum $j$-value, then creates a counter to keep track of what row it is on, where each row corresponds to a different $j$-value. The script finds all $k$-values for each $j$-value, then loads these into a vector. It finds the first $j$ odd integers, then applies the built-in nchoosek function to find every possible combination of $q$ integers from set of the first $j$ consecutive odd numbers, where $q<j$. The script loads the combinations into a matrix. We repeat this process for every $q$ smaller than $j$, where the script checks each combination of $q$ odd integers to see if they sum to the product of $k$ and $j$. This is because the first $j$ consecutive odd integers are numerators from the series for $j$, thus if the combination is valid, it will sum to the product of $k$ and $j$. Each time the script finds a combination that works, it increases a counter by one, then prints the total number of combinations for each $k$-value in the appropriate position in a number table. All loops are while loops. We apply a similar process to produce tables for $r_h(x)$, except the user is asked to input $h$, where $h$ is the fixed number of terms per partition. The program then only allows values of $q$ equal to this number. 

Note that in the tables below, nonzero terms are written in bold to make the patterns easier to read. Values for $j$ are given in the column on the far left; values for $k$ are given across the first row. 

\subsection{Restriction to three terms:}
\begin{table}[h!]
  \centering
    \begin{tabular}{ |c c|c|c|c|c|c|c|c|c|c|c|c|c|c|c|}
    \hline
    \multicolumn{16}{|c|}{Values for $f_{O_j,3}(k)$} \\
    \hline
    $j$ $\hspace{0.2 cm} | $& $k$ & $1$ & $2$ & $3$ & $4$ & $5$ & $6$ & $7$ & $8$ & $9$ & $10$ & $11$ & $12$ & $13$ & $14$\\
    \hline\hline
    3 & $ $ & 0 & 0 & 0 & 0 & 0 & 0 & 0 & 0 & 0 & 0 & 0 & 0 & 0 & 0\\
    
    4 & $ $ & 0 & 0 & 0 & 0 & 0 & 0 & 0 & 0 & 0 & 0 & 0 & 0 & 0 & 0\\
    
    5 & $ $ & 0 & 0 & \textbf{2} & 0 & 0 & 0 & 0 & 0 & 0 & 0 & 0 & 0 & 0 & 0\\
    
    6 & $ $ & 0 & 0 & 0 & 0 & 0 & 0 & 0 & 0 & 0 & 0 & 0 & 0 & 0 & 0\\
    
    7 & $ $ & 0 & 0 & \textbf{5} & 0 & 0 & 0 & 0 & 0 & 0 & 0 & 0 & 0 & 0 & 0\\
    
    8 & $ $ & 0 & 0 & 0 & 0 & 0 & 0 & 0 & 0 & 0 & 0 & 0 & 0 & 0 & 0\\
   
    9 & $ $ & \textbf{1} & 0 & \textbf{8} & 0 & \textbf{1} & 0 & 0 & 0 & 0 & 0 & 0 & 0 & 0 & 0\\
  
    10 & $ $ & 0 & 0 & 0 & 0 & 0 & 0 & 0 & 0 & 0 & 0 & 0 & 0 & 0 & 0\\
    
    11 & $ $ & \textbf{1} & 0 &\textbf{13} & 0 & \textbf{1} & 0 & 0 & 0 & 0 & 0 & 0 & 0 & 0 & 0\\
    
    12 & $ $ & 0 & 0 & 0 & 0 & 0 & 0 & 0 & 0 & 0 & 0 & 0 & 0 & 0 & 0\\
   
    13 & $ $ & \textbf{2} & 0 & \textbf{18} & 0 & \textbf{2} & 0 & 0 & 0 & 0 & 0 & 0 & 0 & 0 & 0\\
    
    14 & $ $ & 0 & 0 & 0 & 0 & 0 & 0 & 0 & 0 & 0 & 0 & 0 & 0 & 0 & 0\\
    
    15 & $ $ & \textbf{3} & 0 & \textbf{25} & 0 & \textbf{3} & 0 & 0 & 0 & 0 & 0 & 0 & 0 & 0 & 0\\
    \hline
    \end{tabular}
\end{table}

\newpage

\subsection{Restriction to four terms:}
\begin{table}[h!]
  \centering
    \begin{tabular}{ |c c|c|c|c|c|c|c|c|c|c|c|c|c|c|c|}
    \hline
    \multicolumn{16}{|c|}{Values for $f_{O_j,4}(k)$} \\
    \hline
    $j$ $\hspace{0.2 cm} | $& $k$ & $1$ & $2$ & $3$ & $4$ & $5$ & $6$ & $7$ & $8$ & $9$ & $10$ & $11$ & $12$ & $13$ & $14$\\
    \hline\hline
    3 & $ $ & 0 & 0 & 0 & 0 & 0 & 0 & 0 & 0 & 0 & 0 & 0 & 0 & 0 & 0\\
    
    4 & $ $ & 0 & 0 & 0 & 0 & 0 & 0 & 0 & 0 & 0 & 0 & 0 & 0 & 0 & 0\\
    
    5 & $ $ & 0 & 0 & 0 & \textbf{1} & 0 & 0 & 0 & 0 & 0 & 0 & 0 & 0 & 0 & 0\\
    
    6 & $ $ & 0 & 0 & \textbf{1} & \textbf{3} & \textbf{1} & 0 & 0 & 0 & 0 & 0 & 0 & 0 & 0 & 0\\
    
    7 & $ $ & 0 & 0 & 0 & \textbf{5} & 0 & 0 & 0 & 0 & 0 & 0 & 0 & 0 & 0 & 0\\
    
    8 & $ $ & 0 & \textbf{1} & \textbf{5} & \textbf{8} & \textbf{5} & \textbf{1} & 0 & 0 & 0 & 0 & 0 & 0 & 0 & 0\\
   
    9 & $ $ & 0 & \textbf{1} & 0 & \textbf{12} & 0 & \textbf{1} & 0 & 0 & 0 & 0 & 0 & 0 & 0 & 0\\
  
    10 & $ $ & 0 & \textbf{2} & \textbf{10} & \textbf{18} & \textbf{10} & \textbf{2} & 0 & 0 & 0 & 0 & 0 & 0 & 0 & 0\\
    
    11 & $ $ & 0 & \textbf{3} & 0 & \textbf{24} & 0 & \textbf{3} & 0 & 0 & 0 & 0 & 0 & 0 & 0 & 0\\
    
    12 & $ $ & 0 & \textbf{5} & \textbf{21} & \textbf{33} & \textbf{21} & \textbf{5} & 0 & 0 & 0 & 0 & 0 & 0 & 0 & 0\\
   
    13 & $ $ & 0 & \textbf{6} & 0 & \textbf{43} & 0 & \textbf{6} & 0 & 0 & 0 & 0 & 0 & 0 & 0 & 0\\
    
    14 & $ $ & 0 & \textbf{9} & \textbf{35} & \textbf{55} & \textbf{35} & \textbf{9} & 0 & 0 & 0 & 0 & 0 & 0 & 0 & 0\\
    
    15 & $ $ & 0 & \textbf{11} & 0 & \textbf{69} & 0 & \textbf{11} & 0 & 0 & 0 & 0 & 0 & 0 & 0 & 0\\
    \hline
    \end{tabular}
\end{table}

\subsection{Restriction to five terms:}
\begin{table}[h!]
  \centering
    \begin{tabular}{ |c c|c|c|c|c|c|c|c|c|c|c|c|c|c|c|}
    \hline
    \multicolumn{16}{|c|}{Values for $f_{O_j,5}(k)$} \\
    \hline
    $j$ $\hspace{0.2 cm} | $& $k$ & $1$ & $2$ & $3$ & $4$ & $5$ & $6$ & $7$ & $8$ & $9$ & $10$ & $11$ & $12$ & $13$ & $14$\\
    \hline\hline
    3 & $ $ & 0 & 0 & 0 & 0 & 0 & 0 & 0 & 0 & 0 & 0 & 0 & 0 & 0 & 0\\
    
    4 & $ $ & 0 & 0 & 0 & 0 & 0 & 0 & 0 & 0 & 0 & 0 & 0 & 0 & 0 & 0\\
    
    5 & $ $ & 0 & 0 & 0 & 0 & 0 & 0 & 0 & 0 & 0 & 0 & 0 & 0 & 0 & 0\\
    
    6 & $ $ & 0 & 0 & 0 & 0 & 0 & 0 & 0 & 0 & 0 & 0 & 0 & 0 & 0 & 0\\
    
    7 & $ $ & 0 & 0 & 0 & 0 & \textbf{3} & 0 & 0 & 0 & 0 & 0 & 0 & 0 & 0 & 0\\
    
    8 & $ $ & 0 & 0 & 0 & 0 & 0 & 0 & 0 & 0 & 0 & 0 & 0 & 0 & 0 & 0\\
   
    9 & $ $ & 0 & 0 & \textbf{1} & 0 & \textbf{12} & 0 & \textbf{1} & 0 & 0 & 0 & 0 & 0 & 0 & 0\\
  
    10 & $ $ & 0 & 0 & 0 & 0 & 0 & 0 & 0 & 0 & 0 & 0 & 0 & 0 & 0 & 0\\
    
    11 & $ $ & 0 & 0 & \textbf{5} & 0 & \textbf{32} & 0 & \textbf{5} & 0 & 0 & 0 & 0 & 0 & 0 & 0\\
    
    12 & $ $ & 0 & 0 & 0 & 0 & 0 & 0 & 0 & 0 & 0 & 0 & 0 & 0 & 0 & 0\\
   
    13 & $ $ & 0 & 0 & \textbf{13} & 0 & \textbf{73} & 0 & \textbf{13} & 0 & 0 & 0 & 0 & 0 & 0 & 0\\
    
    14 & $ $ & 0 & 0 & 0 & 0 & 0 & 0 & 0 & 0 & 0 & 0 & 0 & 0 & 0 & 0\\
    
    15 & $ $ & 0 & 0 & \textbf{30} & 0 & \textbf{141} & 0 & \textbf{30} & 0 & 0 & 0 & 0 & 0 & 0 & 0\\
    \hline
    \end{tabular}
\end{table}

\newpage

\subsection{Restriction to six terms:}
\begin{table}[h!]
  \centering
    \begin{tabular}{ |c c|c|c|c|c|c|c|c|c|c|c|c|c|c|c|}
    \hline
    \multicolumn{16}{|c|}{Values for $f_{O_j,6}(k)$} \\
    \hline
    $j$ $\hspace{0.2 cm} | $& $k$ & $1$ & $2$ & $3$ & $4$ & $5$ & $6$ & $7$ & $8$ & $9$ & $10$ & $11$ & $12$ & $13$ & $14$\\
    \hline\hline
    3 & $ $ & 0 & 0 & 0 & 0 & 0 & 0 & 0 & 0 & 0 & 0 & 0 & 0 & 0 & 0\\
    
    4 & $ $ & 0 & 0 & 0 & 0 & 0 & 0 & 0 & 0 & 0 & 0 & 0 & 0 & 0 & 0\\
    
    5 & $ $ & 0 & 0 & 0 & 0 & 0 & 0 & 0 & 0 & 0 & 0 & 0 & 0 & 0 & 0\\
    
    6 & $ $ & 0 & 0 & 0 & 0 & 0 & 0 & 0 & 0 & 0 & 0 & 0 & 0 & 0 & 0\\
    
    7 & $ $ & 0 & 0 & 0 & 0 & 0 & \textbf{1} & 0 & 0 & 0 & 0 & 0 & 0 & 0 & 0\\
    
    8 & $ $ & 0 & 0 & 0 & 0 & \textbf{2} & \textbf{4} & \textbf{2} & 0 & 0 & 0 & 0 & 0 & 0 & 0\\
   
    9 & $ $ & 0 & 0 & 0 & \textbf{1} & 0 & \textbf{8} & 0 & \textbf{1} & 0 & 0 & 0 & 0 & 0 & 0\\
  
    10 & $ $ & 0 & 0 & 0 & \textbf{2} & \textbf{10} & \textbf{18} & \textbf{10} & \textbf{2} & 0 & 0 & 0 & 0 & 0 & 0\\
    
    11 & $ $ & 0 & 0 & 0 & \textbf{5} & 0 & \textbf{32} & 0 & \textbf{5} & 0 & 0 & 0 & 0 & 0 & 0\\
    
    12 & $ $ & 0 & 0 & \textbf{1} & \textbf{11} & \textbf{39} & \textbf{58} & \textbf{39} & \textbf{11} & \textbf{1} & 0 & 0 & 0 & 0 & 0\\
   
    13 & $ $ & 0 & 0 & 0 & \textbf{19} & 0 & \textbf{94} & 0 & \textbf{19} & 0 & 0 & 0 & 0 & 0 & 0\\
    
    14 & $ $ & 0 & 0 & \textbf{3} & \textbf{33} & \textbf{103} & \textbf{151} & \textbf{103} & \textbf{33} & \textbf{3} & 0 & 0 & 0 & 0 & 0\\
    
    15 & $ $ & 0 & 0 & 0 & \textbf{54} & 0 & \textbf{227} & 0 & \textbf{54} & 0 & 0 & 0 & 0 & 0 & 0\\
    \hline
    \end{tabular}
\end{table}

\subsection{Restriction to seven terms:}
\begin{table}[h!]
  \centering
    \begin{tabular}{ |c c|c|c|c|c|c|c|c|c|c|c|c|c|c|c|}
    \hline
    \multicolumn{16}{|c|}{Values for $f_{O_j,7}(k)$} \\
    \hline
    $j$ $\hspace{0.2 cm} | $& $k$ & $1$ & $2$ & $3$ & $4$ & $5$ & $6$ & $7$ & $8$ & $9$ & $10$ & $11$ & $12$ & $13$ & $14$\\
    \hline\hline
    3 & $ $ & 0 & 0 & 0 & 0 & 0 & 0 & 0 & 0 & 0 & 0 & 0 & 0 & 0 & 0\\
    
    4 & $ $ & 0 & 0 & 0 & 0 & 0 & 0 & 0 & 0 & 0 & 0 & 0 & 0 & 0 & 0\\
    
    5 & $ $ & 0 & 0 & 0 & 0 & 0 & 0 & 0 & 0 & 0 & 0 & 0 & 0 & 0 & 0\\
    
    6 & $ $ & 0 & 0 & 0 & 0 & 0 & 0 & 0 & 0 & 0 & 0 & 0 & 0 & 0 & 0\\
    
    7 & $ $ & 0 & 0 & 0 & 0 & 0 & 0 & 0 & 0 & 0 & 0 & 0 & 0 & 0 & 0\\
    
    8 & $ $ & 0 & 0 & 0 & 0 & 0 & 0 & 0 & 0 & 0 & 0 & 0 & 0 & 0 & 0\\
   
    9 & $ $ & 0 & 0 & 0 & 0 & 0 & 0 & \textbf{4} & 0 & 0 & 0 & 0 & 0 & 0 & 0\\
  
    10 & $ $ & 0 & 0 & 0 & 0 & 0 & 0 & 0 & 0 & 0 & 0 & 0 & 0 & 0 & 0\\
    
    11 & $ $ & 0 & 0 & 0 & 0 & \textbf{3} & 0 & \textbf{24} & 0 & \textbf{3} & 0 & 0 & 0 & 0 & 0\\
    
    12 & $ $ & 0 & 0 & 0 & 0 & 0 & 0 & 0 & 0 & 0 & 0 & 0 & 0 & 0 & 0\\
   
    13 & $ $ & 0 & 0 & 0 & 0 & \textbf{19} & 0 & \textbf{94} & 0 & \textbf{19} & 0 & 0 & 0 & 0 & 0\\
    
    14 & $ $ & 0 & 0 & 0 & 0 & 0 & 0 & 0 & 0 & 0 & 0 & 0 & 0 & 0 & 0\\
    
    15 & $ $ & 0 & 0 & 0 & 0 & \textbf{70} & 0 & \textbf{289} & 0 & \textbf{70} & 0 & 0 & 0 & 0 & 0\\
    \hline
    \end{tabular}
\end{table}

\newpage

\subsection{Restriction to eight terms:}
\begin{table}[h!]
  \centering
    \begin{tabular}{ |c c|c|c|c|c|c|c|c|c|c|c|c|c|c|c|}
    \hline
    \multicolumn{16}{|c|}{Values for $f_{O_j,8}(k)$} \\
    \hline
    $j$ $\hspace{0.2 cm} | $& $k$ & $1$ & $2$ & $3$ & $4$ & $5$ & $6$ & $7$ & $8$ & $9$ & $10$ & $11$ & $12$ & $13$ & $14$\\
    \hline\hline
    3 & $ $ & 0 & 0 & 0 & 0 & 0 & 0 & 0 & 0 & 0 & 0 & 0 & 0 & 0 & 0\\
    
    4 & $ $ & 0 & 0 & 0 & 0 & 0 & 0 & 0 & 0 & 0 & 0 & 0 & 0 & 0 & 0\\
    
    5 & $ $ & 0 & 0 & 0 & 0 & 0 & 0 & 0 & 0 & 0 & 0 & 0 & 0 & 0 & 0\\
    
    6 & $ $ & 0 & 0 & 0 & 0 & 0 & 0 & 0 & 0 & 0 & 0 & 0 & 0 & 0 & 0\\
    
    7 & $ $ & 0 & 0 & 0 & 0 & 0 & 0 & 0 & 0 & 0 & 0 & 0 & 0 & 0 & 0\\
    
    8 & $ $ & 0 & 0 & 0 & 0 & 0 & 0 & 0 & 0 & 0 & 0 & 0 & 0 & 0 & 0\\
   
    9 & $ $ & 0 & 0 & 0 & 0 & 0 & 0 & 0 & \textbf{1} & 0 & 0 & 0 & 0 & 0 & 0\\
  
    10 & $ $ & 0 & 0 & 0 & 0 & 0 & 0 & \textbf{2} & \textbf{5} & \textbf{2} & 0 & 0 & 0 & 0 & 0\\
    
    11 & $ $ & 0 & 0 & 0 & 0 & 0 & \textbf{1} & 0 & \textbf{13} & 0 & \textbf{1} & 0 & 0 & 0 & 0\\
    
    12 & $ $ & 0 & 0 & 0 & 0 & 0 & \textbf{5} & \textbf{21} & \textbf{33} & \textbf{21} & \textbf{5} & 0 & 0 & 0 & 0\\
   
    13 & $ $ & 0 & 0 & 0 & 0 & 0 & \textbf{13} & 0 & \textbf{73} & 0 & \textbf{13} & 0 & 0 & 0 & 0\\
    
    14 & $ $ & 0 & 0 & 0 & 0 & \textbf{3} & \textbf{33} & \textbf{103} & \textbf{151} & \textbf{103} & \textbf{33} & \textbf{3} & 0 & 0 & 0\\
    
    15 & $ $ & 0 & 0 & 0 & 0 & 0 & \textbf{70} & 0 & \textbf{289} & 0 & \textbf{70} & 0 & 0 & 0 & 0\\
    \hline
    \end{tabular}
\end{table}

Upon inspection, we see that sequences such as the third column for restriction to four terms and the fifth column for restriction to five terms are not listed in the OEIS. 

\section{Further computational results:}
\label{sec:9}
\subsection{Partitioning series with even numerators into integers}

Using a MATLAB script, we studied the series corresponding to integral points on the surface $t=\frac{x^2 +x}{y}$ to find what integers said series can be partitioned into. Not all can be partitioned into every integer smaller than $t$, though this is frequently the case, especially for longer series. Often there are multiple combinations for a given integer. We include the results for only three series below, as most others would require too much space to be practical to show. 

\noindent Series:
\begin{equation*}
    6=\frac{2}{2}+\frac{4}{2}+\frac{6}{2}
\end{equation*}
Partitions:
\begin{displaymath}
1=\frac{2}{2} \hspace{0.5cm} 2=\frac{4}{2} \hspace{0.5cm} 3=\frac{6}{2} \hspace{0.5cm}  3=\frac{2}{2}+\frac{4}{2} \hspace{0.5cm} 4=\frac{2}{2}+\frac{6}{2} \hspace{0.5cm} 5=\frac{4}{2}+\frac{6}{2}
\end{displaymath}
Series:
\begin{equation*}
10=\frac{2}{2}+\frac{4}{2}+\frac{6}{2}+\frac{8}{2}   
\end{equation*}
Partitions:
\begin{displaymath}
1=\frac{2}{2} \hspace{0.5cm} 2=\frac{4}{2} \hspace{0.5cm} 3=\frac{6}{2} \hspace{0.5cm}  3=\frac{2}{2}+\frac{4}{2} \hspace{0.5cm} 4=\frac{8}{2} \hspace{0.5cm} 4=\frac{2}{2}+\frac{6}{2} \hspace{0.5cm} 5=\frac{2}{2}+\frac{8}{2} \hspace{0.5cm} 5=\frac{4}{2}+\frac{6}{2}
\end{displaymath}
\begin{displaymath}
6=\frac{4}{2}+\frac{8}{2} \hspace{0.5cm} 6=\frac{2}{2}+\frac{4}{2}+\frac{6}{2} \hspace{0.5cm} 7=\frac{6}{2}+\frac{8}{2} \hspace{0.5cm} 7=\frac{2}{2}+\frac{4}{2}+\frac{8}{2} \hspace{0.5cm} 8=\frac{2}{2}+\frac{6}{2}+\frac{8}{2} \hspace{0.5cm} 9=\frac{4}{2}+\frac{6}{2}+\frac{8}{2}
\end{displaymath}
Series:
\begin{equation*}
10=\frac{2}{3}+\frac{4}{3}+\frac{6}{3}+\frac{8}{3}+\frac{10}{3} 
\end{equation*}
Partitions:
\begin{displaymath}
2=\frac{6}{3} \hspace{0.5cm} 2=\frac{2}{3}+\frac{4}{3} \hspace{0.5cm} 4=\frac{2}{3}+\frac{10}{3} \hspace{0.5cm} 4=\frac{4}{3}+\frac{8}{3} \hspace{0.5cm} 4=\frac{2}{3}+\frac{4}{3}+\frac{6}{3} \hspace{0.5cm} 6=\frac{8}{3}+\frac{10}{3}  
\end{displaymath}
\begin{displaymath}
6=\frac{2}{3}+\frac{6}{3}+\frac{10}{3} \hspace{0.5cm} 6=\frac{4}{3}+\frac{6}{3}+\frac{8}{3} \hspace{0.5cm} 8=\frac{6}{3}+\frac{8}{3}+\frac{10}{3} \hspace{0.5cm} 8=\frac{2}{3}+\frac{4}{3}+\frac{8}{3}+\frac{10}{3}
\end{displaymath}

\section{Connection to Diophantine Equations}
\label{sec:10}

Diophantine equations have proved interesting in part due to the difficulty in predicting when integral solutions occur, and if so, how many exist. Matiyasevich proved that there is no general algorithm to predict when these solutions occur \cite{thirteen}, though that does not mean we cannot work with specific cases. Andrew Wiles' proof of Fermat's Last Theorem is a famous example of this fact \cite{{fourteen},{fifteen}}. Just as even numerator series can be connected to finding integral points on a surface that is really a nonlinear Diophantine equation, odd numerator series can also be connected to such an equation. The case for odd numerators is different though, as it contains a variable number of unknowns. 

We find that it is possible to predict whether the equation
$2(n_1+n_2+ \cdots +n_{h-1}+n_h)-(kj+h)=0$ can have integral solutions. We see that in certain situations we can also predict how many there are. This particular class of equations is interesting not only because it is nonlinear, but also because the total number of unknowns is dependent on the value of the variable $h$. 
\begin{proposition}
\label{prp:10.1}
Suppose we have a Diophantine equation of the form 
\begin{equation}
2(n_1+n_2+ \cdots + n_{h-1}+n_h)-(kj+h)=0
\end{equation}

where $n_1$, $n_2,\ldots, n_{h-1}$, $n_h$, $j$, $k$, and $h$ are variables with positive integral values, 
\begin{displaymath}
k<j
\end{displaymath}
\begin{center}
$n_1, n_2, \cdots, n_{h-1},$ and $n_h$ are distinct 
\end{center}
\begin{displaymath}
n_1, n_2, \ldots, n_{h-1}, n_h \leq j
\end{displaymath}
There is never a solution if $j<3$ or $k\geq2h$. If $0<k<2h$ and $j>2$, there can be a solution, but it is not guaranteed. If there are $M$ solutions when using a given $k_1$  and $j$, where $k_1<h<j$, then there are also $M$ solutions for $k_2=2h-k_1$, where $j$ is fixed.
\end{proposition}

\begin{proof}
When we fix the number of terms per combination to $h$, we are really working with a Diophantine equation of the form:
\begin{equation}
    \left(\frac{2n_1 - 1}{j} + \frac{2n_2 - 1}{j} + \cdots + \frac{2n_{h-1} - 1}{j} + \frac{2n_h - 1}{j} \right)=k
\end{equation}
\begin{equation}
    \frac{2\bigr(n_1+n_2+ \cdots +n_{h-1}+n_h\bigr)-h}{j}=k
\end{equation}
\begin{center}
    $n_1, n_2, \cdots ,n_{h-1},$ and $n_h$ are distinct.
\end{center}
\begin{displaymath}
    n_1, n_2, \ \ldots, \ n_{h-1}, n_h \leq j
\end{displaymath}
\begin{displaymath}
    k<j, \ j>2
\end{displaymath}

We can rewrite as:
\begin{equation}
    2(n_1+n_2+ \cdots +n_{h-1}+n_h)-(kj+h)=0
\end{equation}
This equation is then nonlinear because it includes the product of two variables. 

Proposition \ref{prp:4.1} stipulates that we must have $j>2$ and $k<j$ for there to be any combinations for $k$, regardless of the number of terms used. In Section \ref{sec:7} we proved that there is never a solution if $k\geq2h$, where $h>1$. This is also true when $h=1$. All numerators are less than or equal to $2j-1$, hence no $k \geq 2$ can be found using one term. We know that $k$ must be positive, as all fractions in the series for $j$ are positive. We thus know that for there to be a possibility of a solution, $0<k<2h$. In our proof for symmetry, we found that if there are $M$ solutions for $k_1$, where $k_1<h<j$, then there are also $M$ solutions for $k_2=2h-k_1$. Using this result, we then have
$f_{O_j,h}(k_1)=f_{O_j,h}(k_2)$.
\end{proof}

{\footnotesize  
\medskip
\medskip
\vspace*{1mm} 
 
\noindent {\it Zachary Hoelscher}\\  
Virginia Polytechnic Institute and State University\\
Blacksburg, Virginia 24061\\
E-mail: {\tt zacharyh22@vt.edu}\\ \\  

\noindent {\it Eyvindur Palsson}\\  
Virginia Polytechnic Institute and State University \\
Blacksburg, Virginia 24061\\
E-mail: {\tt palsson@vt.edu}\\ \\

}


\newpage 

{\footnotesize
\begin{thebibliography}{16}

\bibitem{one} A. Sills, Rademacher-type formulas for restricted partition and overpartition functions, {\it Ramanujan J.}, {\bf 23} (2010), 253--264, available online at the URL: \textcolor{blue}{\href{https://doi.org/10.1007/s11139-009-9184-y}{\tt https://doi.org/10.1007/s11139-009-9184-y}}

\bibitem{two} S. Ramanujan, Some properties of p(n), the number of partitions of n, {\it Math. Proc. Cambridge Philos. Soc.}, {\bf 19} (1919), 207--210.

\bibitem{three} M. Bidar, Partition of an integer into distinct bounded parts, identities and bounds, {\it Integers}, {\bf 12} (2012), available online at the URL: \textcolor{blue}{\href{http://math.colgate.edu/~integers/m8/m8.pdf}{\tt http://math.colgate.edu/~integers/m8/m8.pdf}}

\bibitem{four}
P. Hagis, Jr, Partitions into odd and unequal parts, {\it Amer. J. Math.}, {\bf 86} (1964), 317--324, available online at the URL: \textcolor{blue}{\href{https://www.jstor.com/stable/2373167}{\tt https://www.jstor.com/stable/2373167}}

\bibitem{five}
D. Christopher, Partitions with fixed number of sizes, {\it J. Integer Seq.}, {\bf 18} (2015), available online at the URL: \textcolor{blue}{\href{https://cs.uwaterloo.ca/journals/JIS/VOL18/Christopher/chris7.pdf}{\tt https://cs.uwaterloo.ca/journals/JIS/VOL18/Christopher/chris7.pdf}}

\bibitem{six}
B. Riemann, {\"U}ber die anzahl der primzahlen unter einer gegebenen gr{\"o}sse, {\it Monatsberichte der Berliner Akademie}, {\bf 7} (1859), 136--144.

\bibitem{seven}
A. Anggoro, E. Liu, and A. Tulloch, The rascal triangle, {\it College Math. J.}, {\bf 41} (2010), 393--395, available online at the URL: \textcolor{blue}{\href{https://www.maa.org/sites/default/files/Anggoro2010.pdf}{\tt https://www.maa.org/sites/default/files/Anggoro2010.pdf}}

\bibitem{eight}
N.J.A. Sloane, Sequence A077028, {\it The On-Line Encyclopedia of Integer Sequences}, available online at the URL: \textcolor{blue}{\href{https://oeis.org}{\tt https://oeis.org}}

\bibitem{nine}
N.J.A. Sloane, Sequence A004526, {\it The On-Line Encyclopedia of Integer Sequences}, available online at the URL: \textcolor{blue}{\href{https://oeis.org}{\tt https://oeis.org}}

\bibitem{ten}
K. O'Hara, Unimodality of Gaussian coefficients: A constructive proof, \emph{J. Combin. Theory Ser. A}, \textbf{53} (1990), 29-52, available online at the URL: \textcolor{blue}{\href{https://doi.org/10.1016/0097-3165(90)90018-R}{\tt https://doi.org/10.1016/0097-3165(90)90018-R}}

\bibitem{eleven}
M. Aigner, \emph{Graduate Texts in Mathematics: A Course in Enumeration}, Springer-Verlag, (2007).

\bibitem{twelve}
L. T\'{o}th, Counting solutions of quadratic congruences in several variables revisited, {\it J. of Integer Seq.}, {\bf 17} (2014), available online at the URL: \textcolor{blue}{\href{https://cs.uwaterloo.ca/journals/JIS/VOL17/Toth/toth12.html}{\tt https://cs.uwaterloo.ca/journals/JIS/VOL17/Toth/toth12.html}}

\bibitem{thirteen}
Y. Matiyasevich, {\it Hilbert's Tenth Problem}, MIT Press, (1993).

\bibitem{fourteen}
A. Wiles, Modular elliptic curves and Fermat's Last Theorem, {\it Ann. of  Math.}, {\bf 141} (1995), 443--551, available online at the URL: \textcolor{blue}{\href{https://www.jstor.org/stable/2118559}{\tt https://www.jstor.org/stable/2118559}}

\bibitem{fifteen}
A. Wiles and R. Taylor, Ring-theoretic properties of certain Hecke algebras, {\it Ann. of Math.}, {\bf 141} (1995), 553--572, available online at the URL: \textcolor{blue}{\href{https://www.jstor.org/stable/2118560}{\tt https://www.jstor.org/stable/2118560}}

\end{thebibliography}}
\end{document}